\documentclass[11pt]{article}
\usepackage{amsmath,amssymb,amsthm}
\usepackage{natbib}
\usepackage{graphicx}
\usepackage{xcolor}
\usepackage{pgfplots}
\usepackage{subcaption}
\usepackage{float}
\usepackage{nameref}
\usepackage{setspace}
\usepackage{bm}
\usepackage[font=small, labelfont=bf, textfont=it]{caption}
\usepackage[normalem]{ulem}
\usepackage[
    colorlinks=true,
    citecolor=blue,   
    linkcolor=blue,    
    urlcolor=blue     
]{hyperref}
\usepackage[
  a4paper,
  left=3cm,
  right=2.5cm,
  top=2.5cm,
  bottom=2.5cm
]{geometry}

\usepackage{booktabs}
\usepackage{multirow}
\usepackage{etoolbox}
\usepackage{siunitx}
\usepackage{placeins}
\usepackage{needspace}

\renewrobustcmd{\bfseries}{\fontseries{b}\selectfont}
\usepackage{bbm}
\usepackage{enumitem}

\newcommand{\bY}{\boldsymbol{Y}}
\newcommand{\by}{\boldsymbol{y}}
\newcommand{\PP}{\mathbb{P}}
\newcommand{\RR}{\mathbb{R}}
\newcommand{\EE}{\mathbb{E}}

\newcommand{\bmu}{\boldsymbol{\mu}}

\newcommand{\calN}{\mathcal{N}}

\newcommand{\bfbeta}{\boldsymbol{\beta}}

\newtheorem{theorem}{Theorem}[section]
\newtheorem{lemma}[theorem]{Lemma}
\newtheorem{proposition}[theorem]{Proposition}
\newtheorem{corollary}[theorem]{Corollary}

\theoremstyle{definition}

\theoremstyle{remark}

\theoremstyle{remark}
\newtheorem{remark}[theorem]{Remark}

\newtheoremstyle{named}%
  {}{}            
  {\itshape}      
  {}              
  {\bfseries}     
  {.}             
  {.5em}          
  {\thmname{#1}\thmnote{ \textnormal{(#3)}}} 

\theoremstyle{named}

\title{Posterior consistency of P\'olya trees for deconvolution\\ under the linear model}

\author{
Nakul Shenoy$^{*}$
\and
Asaf Weinstein$^{\dagger}$
}

\date{}

\begin{document}

\maketitle

\begingroup
\renewcommand\thefootnote{\fnsymbol{footnote}}

\footnotetext[1]{Department of Statistics, The Hebrew University of Jerusalem. 
E-mail: nakulshn@gmail.com}

\footnotetext[2]{Department of Statistics, The Hebrew University of Jerusalem.
E-mail: asaf.weinstein@mail.huji.ac.il.}
\endgroup

\begin{abstract}
Several recent works have addressed the problem of deconvolution under a linear model, where the goal is to  estimate a completely unknown $G_0$ from a vector of noisy observations $\bY = X\bfbeta + \boldsymbol{\epsilon}$, assuming the coefficients $\beta_j$ are i.i.d.~unobserved realizations from $G_0$. 
Assuming $G_0$ has a density $g_0$, we study theoretically a Bayesian nonparametric method proposed in \cite{weinstein2025nonparametric} that postulates a P\'olya tree prior $\Pi$ on $g_0$ and bases a deconvolution estimate on the posterior distribution $\Pi(\cdot|\bY)$. 
Our main result asserts that under the true model (fixed and unknown $g_0$), and under a suitable condition on the minimum eigenvalue of $X^\top X$, the posterior $\Pi(\cdot|\bY)$ concentrates around $g_0$ in sup-norm. 
The analysis presented builds on and extends results from \cite{castillo2017polya}, where posterior consistency of P\'olya trees was proved for density estimation, the simpler problem of estimating $g_0$ when observing the coefficients $\beta_j$ directly. 

\end{abstract}

\section{Introduction}
\label{sec:intro}

In statistical problems involving a vector $\bfbeta_0 = (\beta_{01},...,\beta_{0n})\in \RR^n$ of unobserved parameters of the same kind, a popular strategy for borrowing strength is to assume that the parameters are i.i.d.~(or at least exchangeable) draws from some common distribution, 
\begin{align}
\label{eq:prior-iid}
\beta_{0j} & \overset{iid}{\sim}G_0, \qquad j=1,...,n, 
\end{align}
but regard the fixed distribution $G_0$  as unknown. 
This conceptually simple idea allows one to exploit the symmetry among the parameters to incorporate regularization (shrinkage) into the given likelihood function, while permitting the {\em form} of regularization to be learned from the data rather than specified in advance. 
To emphasize, we assume throughout that \eqref{eq:prior-iid} is part of the data generating mechanism, i.e., that the true model for the data includes some fixed and unknown distribution $G_0$ from which the true parameter values $\beta_{0j}$ are drawn independently. 
Under this general {\em empirical Bayes} (EB) setup, and assuming $n$ is large, the current article is broadly concerned with  nonparametric estimation of the mixing distribution $G_0$, known as a {\em deconvolution} problem. 

In the statistical literature, deconvolution problems are classically considered  within a {\em sequence} model, where for each of the unknown parameters $\beta_{0j}\in \Theta$ we observe an independent noisy measurement from some common and known  (marginal) likelihood, 
\begin{align}
\label{eq:lik-sequence}
Y_j|\beta_{0j} &\overset{ind}{\sim} f(y_j|\beta_{0j}), \qquad j=1,...,n. 
\end{align}
Under the `separable' two-level model specified by \eqref{eq:prior-iid}-\eqref{eq:lik-sequence}, the pairs $(Y_j, \beta_{0j})$ are i.i.d., in particular  the components of $\bY = (Y_1,...,Y_n)$ are i.i.d.~copies from  
$f_{G_0}(y) := \int f(y\lvert \beta)G_0(d\beta)$, 
a `convolution' of the known likelihood $f$ with the unknown $G_0$.
\footnote{For our purposes we assume $G_0$ is identifiable, although estimating a mixing distribution can in some cases be useful even without identifiability  \citep{greenshtein2025consistent}.} 
The most classical deconvolution method in the sequence model is the nonparametric maximum likelihood estimator \citep[NPMLE,][]{kiefer1956consistency,lindsay1995mixture}, which seeks a maximizer of the (log-) likelihood  
$\ell(G) := \sum_{j=1}^n \log f_{G}(y_j)$, over all possible distributions $G$. 
The NPMLE is known to be discrete with at most $n$ support points \citep{lindsay1983geometry, lindsay1995mixture}, a property that can be leveraged to speed up computation \citep[e.g.,][]{koenker2014convex}, but on the other hand could be considered a disadvantage if $G_0$ is believed to have a density. 
Various other methods for estimating $G_0$, and approximations, both discrete and continuous, to the NPMLE under the sequence model have been proposed over the years, offering  advantages of different kinds \citep[][among others]{wang2007fast, efron2016empirical, newton2002nonparametric}. 

An alternative {\em approach} to producing a deconvolution estimator of the (fixed) distribution $G_0$ in \eqref{eq:prior-iid} is the nonparametric  Bayes approach, which posits another hierarchy in the model by placing a `prior on priors', 
\begin{align}
\label{eq:prior-hier-Bayes}
\beta_j|G & \overset{iid}{\sim}G, \qquad G\sim \Pi, 
\end{align}
where we write $\beta_j$ and $G$ to distinguish from the true coefficients  $\beta_{0j}$ and distribution $G_0$, and where \eqref{eq:lik-sequence} should now be understood as the conditional distribution of the data given $\bfbeta=\bfbeta_0$ (for an arbitrary $\bfbeta_0$). 
Above, $\Pi$ is a completely specified prior with a rich support in the space of distributions on $\Theta$, to accommodate the nonparametric nature of the problem. 
The fully Bayes model specified by  \eqref{eq:prior-hier-Bayes} and \eqref{eq:lik-sequence} automates information sharing between the samples through the calculation of posteriors (not only of $G$ but also of the parameters $\beta_j$, or the predictive posterior of a fresh sample $Y_{n+1}$),  because under \eqref{eq:prior-hier-Bayes} the pairs $(Y_j, \beta_j)$ are not independent, only conditionally independent given $G$. 
As such, one can obtain a deconvolution estimate of the original $G_0$ in \eqref{eq:prior-iid} by summarizing the posterior $\Pi(\cdot|\bY)$ of $G$.
Assuming a sequence model for the data, \citet{antoniak1974mixtures} first proposed to operationalize the nonparametric Bayes framework by taking $\Pi$ to be a mixture of Dirichlet processes;  
\citet{berry1979empirical, ferguson1983bayesian} worked out examples for specific likelihood functions, and \citet{lo1984class} treated the general case. 
Later, \citet{lavine1994more} extended this construction by showing how posterior distributions can still be calculated analytically when $\Pi$ is chosen to be a {\em P\'olya tree} (PT) prior. 
P\'olya trees form a class of prior distributions that generalizes Dirichlet processes while retaining computational tractability. 
One notable advantage of PT priors over Dirichlet processes (or mixtures thereof) is that they can be chosen so that $G$ has a density almost surely, which makes PTs more suitable for modeling continuous distributions $G_0$. 

Recently, a number of papers have appeared that move {\em beyond} the conventional sequence model \eqref{eq:lik-sequence}, and tackle the nonparametric deconvolution problem in a linear regression model, 
\begin{align}
\label{eq:lik-linear}
\bY\lvert \bfbeta_0 \sim \calN(X\bfbeta_0, \sigma^2I_m), 
\end{align}
where $X\in \RR^{m\times n}$ is a fixed and known  covariate matrix whose columns $X_j\in \RR^m$ are assumed linearly independent, and the noise term $\sigma^2>0$ can be regarded as known or unknown. 
Maintaining the assumption \eqref{eq:prior-iid} that the parameters $\beta_{0j}$, now representing the regression coefficients, are i.i.d., results as before in a two-level model with an unknown prior; however, even at the methodological (algorithmic) level, the deconvolution problem becomes substantially harder compared to the sequence case, because under \eqref{eq:lik-linear} the likelihood of each $Y_i$ depends on the entire vector $\bfbeta_0$. 
From the nonparametric empirical Bayes (EB) perspective, which regards $G_0$ as strictly fixed, the NPMLE is still a perfectly sensible estimator of $G_0$, but its implementation---to maximize the marginal log-likelihood of $\bY$ over all $G$---is considerably more difficult under \eqref{eq:lik-linear}, for example it is generally not a convex problem anymore. 
To circumvent the intractability of the NPMLE, \citet{kim2024flexible} proposed to model $G_0$ as a scale mixture of zero-mean Gaussians (more precisely, they model the rescaled coefficients $\beta_{0j}/\sigma$ as i.i.d.), and employ variational Bayes  techniques that approximate the posterior within the family of product  (posterior) distributions. 
Compared to the variational approach of \citet{carbonetto2012scalable}, this simplifies the computation while using a much more flexible family of priors. 
Essentially, \citet{kim2024flexible} show how in optimizing the variational objective their approach allows to benefit from reductions to the much simpler normal {\em sequence}  model, the special case obtained in \eqref{eq:lik-linear} when $m=n$ and $x_{ij} = 1(i=j), 1\leq i,j\leq n$. 
Encouraged by the computational advantages, \citet{mukherjee2023mean} carried out a theoretical analysis of (a variant of) the variational EB method of Kim et al.~in an asymptotic setting where $m,n\to \infty$, and under some mild assumptions established consistency of the estimator, as well as of the exact NPMLE. 
Still, because the advocated (variational) method uses the naive mean-field approximation, it can be inaccurate if covariates are strongly correlated. 
Intended for those situations, \citet{fan2023gradient} assume a bounded $G_0$ as in Mukherjee et~al., but propose a more ambitious method for approximating the NPMLE, which uses the variational representation while avoiding the mean-field approximation of the posterior. 
Their estimator employs modern techniques and is characterized by a system of gradient flow equations which are solved with an MCMC expectation-maximization approach. 
The authors prove consistency of the estimator to the NPMLE in a high dimensional setting, under conditions on $X$ that build on but relax those of Mukherjee et~al. 

A different, nonparametric Bayes approach to the deconvolution problem in the linear model (and more generally, in {generalized} linear models) was presented in \citet{weinstein2025nonparametric}, who proposed to extend the ideas of \citet{lavine1994more} from the sequence model to regression models. 
Thus, assuming the data now truly follows \eqref{eq:lik-linear} and \eqref{eq:prior-iid} for a fixed and unknown $G_0$, they replace \eqref{eq:prior-iid} with the {\em working assumption} \eqref{eq:prior-hier-Bayes}, where $\Pi$ is a (truncated) P\'olya tree prior on distributions. 
To handle the more complicated likelihood compared with the sequence model, they use an MCMC approach along with a Gibbs sampler that exploits a known  conjugacy property of PTs, specifically, that under \eqref{eq:prior-hier-Bayes} the posterior of $G$ given the {unobserved} parameters $\bfbeta$ is again a P\'olya tree. 
Weinstein et~al.~report encouraging simulation results  both in a (mixed-effects) linear regression example and in logistic regression examples, demonstrating that in terms of estimating the CDF of the true $G_0$, their hierarchical Bayes method improves significantly over different competitors, parametric and non-parametric. 

These encouraging empirical results motivate us in the current article to carry out a theoretical analysis of the method, which is not included in the original work of \citet{weinstein2025nonparametric}. 
Specifically, in a suitable high dimensional setting and assuming the pair $(\bY, \bfbeta_0)$ follows the frequentist (in terms of $G_0$) model given by \eqref{eq:lik-linear} and \eqref{eq:prior-iid}, we set out to establish {\em posterior consistency} of $G$ under the postulated PT model, i.e., that the conditional distribution $\Pi(\cdot\lvert \bY)$ of $G$ converges in a certain sense to a point mass at $G_0$. 
Our work builds on and generalizes results from  \cite{castillo2017polya}, where posterior consistency and corresponding rates for PTs are established in the {density estimation} problem, the setting in which the density of a continuous $G_0$ in \eqref{eq:prior-iid} is to be estimated when observing the coefficients $\bfbeta$ directly without noise. 
By contrast, we consider an empirical Bayes setting  where only the noisy observations $Y_i$ (and $X$) from \eqref{eq:lik-linear} are available, and $\bfbeta$ is latent. 

While posterior consistency of $G$ and rates of convergence for Bayesian nonparametric methods in empirical Bayes problems (where $\bfbeta$ is unobserved) have been studied before \citep{rousseau2024wasserstein}, previous work is, to the best of our knowledge, limited to the sequence model; 
as noted before, the regression model is considerably different and more difficult to analyze, because the likelihood of each $Y_i$ in \eqref{eq:lik-linear} involves the entire (common) vector $\bfbeta$, as opposed to the sequence model in which each $Y_i$ has its own parameter $\beta_{i}$. 
Our main theoretical result shows that, assuming $G_0$ has a density $g_0$, and in an appropriate asymptotic regime where $m,n\to\infty$ and the design matrix is sufficiently well-conditioned, the posterior for the density $g$ of $G$ under the postulated model concentrates around the true density $g_0$ in \emph{sup-norm}. 
As in \citet{castillo2017polya}, the contraction rate reflects a bias--variance tradeoff, which is accounted for by the approximation error from truncating the P\'olya tree on the one hand and the stochastic error from estimating bin probabilities from $n$ samples on the other hand. 
However, as opposed to the problem analyzed in \citet{castillo2017polya}, in our setting the admissible truncation depth, and hence the achievable resolution, is further constrained by the difficulty of recovering $\bfbeta$ from the noisy observations $\bY$ in \eqref{eq:lik-linear}. 
When this additional constraint is inactive, there is no loss relative to the idealized setting in which the coefficients are observed directly, and the rate matches that of \citet{castillo2017polya}.

The rest of the paper is organized as follows. 
In Section~\ref{sec:model} we set up the problem formally and state some basic assumptions. 
Section~\ref{sec:pt} recalls the P\'olya tree-based hierarchical model from \citet{weinstein2025nonparametric}.
In Section~\ref{sec:main-results} we prove the main result of the paper, Theorem~\ref{thm:main}, which asserts posterior concentration of the PT under the true (unknown and nonrandom) density of the coefficients. 
We supplement the main theorem with a corollary showing consistency, in squared $\ell_2$ loss, of the posterior mean estimator of the coefficient vector $\bfbeta_0$, as well as a result for random Gaussian designs, in which case we obtain a simple growth condition on $m,n$ ensuring posterior convergence.
Section~\ref{sec:proof-sketch} contains a high-level outline of the proof of the main result.

\section{Data-generating model and problem setup}
\label{sec:model}

The vector $\bY = (Y_1,...,Y_m)$ of observations is assumed to be generated from the linear model \eqref{eq:lik-linear}, where $m\ge n$, $X\in\RR^{m\times n}$ is a fixed design matrix with full column rank, and $\sigma^2>0$ is taken to be known. We allow $\sigma^2$ to depend on $(m,n)$, although this dependence is suppressed in the notation.
Additionally, the components of the unobserved coefficient vector $\bfbeta_0=(\beta_{01},\dots,\beta_{0n})\in\RR^n$ are assumed to be i.i.d.~draws as specified in \eqref{eq:prior-iid}, where $G_0$ is fixed and unknown. 
Throughout, we assume that $G_0$ admits a density $g_0$ supported on a known compact interval $[a,b]\subset\RR$. We assume that $g_0$ is $\alpha$-H\"older smooth for some $\alpha\in(0,1]$, and is bounded above and below away from zero, i.e., that there exist constants $L_0<\infty$ and $0<m_0\le M_0<\infty$ such that
\[
m_0 \le g_0(x)\le M_0,
\qquad
|g_0(x)-g_0(y)|\le L_0 |x-y|^\alpha,
\qquad x,y\in[a,b].
\]
For every $g_0$, the model given by \eqref{eq:prior-iid} and \eqref{eq:lik-linear} defines a joint distribution on $(\bfbeta_0,\bY)$, which we denote by $\PP_{g_0}$. This is the {\em true} data-generating law.
Expectations under $\PP_{g_0}$ are correspondingly denoted by $\EE_{g_0}$.

Following \citet{weinstein2025nonparametric}, to draw inference on $g_0$, we postulate the hierarchical Bayes model given by \eqref{eq:prior-hier-Bayes} and \eqref{eq:lik-linear}, specifying a truncated P\'olya tree prior of depth $L$ on the unknown mixing density $g$. 
For each fixed $\by\in \RR^m$, we will use $\Pi_L(\cdot\mid \by)$ to denote the posterior of $g$ under this postulated model, and $\Pi_L(\cdot\mid \bY)$ for the corresponding random object obtained by plugging $\by = \bY$.
Our main objective in this article is to show that, under suitable conditions on the design matrix $X$, the resulting posterior concentrates around $g_0$ in sup-norm as $m,n\to\infty$. 
Throughout, posterior contraction is always understood with respect to the true law. 
For example, a statement of the form
\begin{equation}
\label{eq:post-contraction}
\EE_{g_0}\!\left[\Pi_L\!\left(\|g-g_0\|_\infty>C_0\varepsilon_n \mid \bY\right)\right]\to 0    
\end{equation}
means that the posterior arising from the postulated model concentrates near the true density $g_0$, when expectation over $\bY$ is taken under $\PP_{g_0}$. 

\section{A hierarchical Bayes approach}
\label{sec:pt}

We recall the truncated P\'olya tree prior from \cite{weinstein2025nonparametric} and the resulting hierarchical Bayes model. 

\subsection{The level-$L$ truncated P\'olya tree prior}

Let $L = L(m,n) \in \mathbb{N}$ denote the truncation level (or {\em depth}) of the tree, to be specified later as a function of $(m,n)$.

\bigskip
\noindent \textbf{Dyadic partition of the support}. 
For each level $l\ge 0$, partition $[a,b]$ into $2^l$ dyadic subintervals
\[
I_k^l :=
\begin{cases}
\left[a+(b-a)\frac{k}{2^l},\, a+(b-a)\frac{k+1}{2^l}\right),
& k=0,\dots,2^l-2,\\[0.4em]
\left[b-\frac{b-a}{2^l},\, b\right],
& k=2^l-1.
\end{cases}
\]
Thus $\{I_k^l:0\le k<2^l\}$ is a partition of $[a,b]$ for every $l$.

\bigskip
\noindent \textbf{Construction of the truncated P\'olya tree prior}. 
For each node $(l,k)$, with $l=0,\dots,L-1$ and $k=0,\dots,2^l-1$, let
\[
V_{l,k}\overset{ind}{\sim}\mathrm{Beta}(1,1)
\]
denote the random split proportion assigned to the left child. These variables determine a random probability distribution $G$ on the level-$L$ partition recursively as follows. 
Starting from $G(I_0^0)=1$, define for each $l=0,\dots,L-1$ and $k=0,\dots,2^l-1$,
\[
G(I_{2k}^{\,l+1}) = V_{l,k}\,G(I_k^l),
\qquad
G(I_{2k+1}^{\,l+1}) = (1-V_{l,k})\,G(I_k^l).
\]
In other words, the mass assigned to a terminal cell $I_k^L$ is the product of the split proportions along the unique path from the root to that cell.
The truncated P\'olya tree prior then identifies $G$ with the piecewise-constant density
\[
g(x)=\sum_{k=0}^{2^L-1}\frac{2^L}{b-a}\,G(I_k^L)\,\mathbbm{1}\{x\in I_k^L\},
\qquad x\in[a,b].
\]
We denote the resulting prior on densities $g$ by $\Pi_L$.

\subsection{Hierarchical Bayes model}

Incorporating the level-$L$ truncated PT prior $\Pi_L$ on $g$ leads to the following hierarchical Bayes model, to which we refer as the {\em working model}:
\begin{align}
&g \sim \Pi_L, \label{eq:pt-hier-1}\\
&\beta_1,...,\beta_n \mid g \overset{iid}{\sim} g
\label{eq:pt-hier-2}\\
&\bY\mid \bfbeta \sim \calN(X\bfbeta,\sigma^2 I_m), 
\label{eq:pt-hier-3}
\end{align}
where we used $\beta_j$ here instead of $\beta_{0j}$ to distinguish the coefficients from those of the true model. 
We rely on the posterior distribution of $g$ given $\bY$ under the working model to estimate the true density $g_0$.

A key feature of the P\'olya tree construction is that the posterior distribution of $g$ given $\bfbeta$ remains a truncated P\'olya tree \citep{ferguson1974prior, mauldin1992polya}. To state this formally, define the empirical bin counts,
\[
N_{\bfbeta}(I_k^l):=\sum_{j=1}^n \mathbbm{1}\{\beta_j\in I_k^l\},
\qquad l=0,\dots,L,\quad k=0,\dots,2^l-1,
\]
and the left/right child counts,
\[
N_{\bfbeta}^{(0)}(I_k^l):=N_{\bfbeta}(I_{2k}^{l+1}),
\qquad
N_{\bfbeta}^{(1)}(I_k^l):=N_{\bfbeta}(I_{2k+1}^{l+1}),
\qquad l=0,\dots,L-1,\quad k=0,\dots,2^l-1.
\]
Then, under the $\mathrm{Beta}(1,1)$ prior, we have 
\[
V_{l,k}\mid \bfbeta \sim \mathrm{Beta}\big(1+N_{\bfbeta}^{(0)}(I_k^l),\,1+N_{\bfbeta}^{(1)}(I_k^l)\big),
\]
independently across nodes $(l,k)$. 
Equivalently, the conditional distribution $\Pi_L(\cdot\mid \bfbeta)$ is again a truncated P\'olya tree, with updated Beta parameters determined by the bin counts of $\bfbeta$.

\subsection{Posterior distribution}
Since under the working model $g$ and $\bY$ are independent given $\bfbeta$, the posterior distribution of $g$ can be written as the mixture
\[
\Pi_L(A\mid \by)
=
\int \Pi_L(A\mid \bfbeta)\,\Pi_{L,\beta}(d\bfbeta\mid \by),
\]
for every measurable set $A$ of densities supported on $[a,b]$. Here, $\Pi_{L,\beta}(\cdot \mid \by)$ denotes the posterior distribution of $\bfbeta$ under \eqref{eq:pt-hier-1}-\eqref{eq:pt-hier-3} (analogously, $\Pi_{L,\beta}$ denotes its prior distribution).

Although the posteriors $\Pi_L(\cdot\mid \by)$ and $\Pi_{L,\beta}(\cdot\mid \by)$ of $g$ and $\bfbeta$ are not available in closed form, posterior inference can be carried out via a Gibbs sampling algorithm, alternating between updates of $\bfbeta$ and $g$. 
Conveniently, the update of $g$ is direct, since $\Pi_L(\cdot |\bfbeta)$ is again a truncated P\'olya tree with Beta parameters updated by the empirical bin counts.

\section{Main results}
\label{sec:main-results}

We now state our main posterior contraction result for the truncated P\'olya tree procedure introduced in the previous section.
Our analysis builds on results of \citet{castillo2017polya}, where posterior contraction for P\'olya tree priors is studied under the \emph{density estimation} setting, i.e., in the case where one directly observes the i.i.d.~samples $\beta_{0j}$ from $g_0$ (in that case, of course, the likelihood of $\bY\lvert \bfbeta$ is irrelevant). 
We therefore begin by recalling the corresponding direct-observation benchmark, adapted to our notation and assumptions. 
Then, we state our main theorem for the regression model, which extends this benchmark to the harder case where the coefficients are latent and only the noisy linear observation $\bY\sim \calN(X\bfbeta_0, \sigma^2I_m)$ is available.

\smallskip
For a truncation level $L\in\mathbb{N}$, define
\begin{equation}
\label{eq:epsilon-def}
\varepsilon_n(L)
:=
(b-a)^{\alpha}2^{-\alpha L}
+
(b-a)^{-1}\sqrt{\frac{L\,2^L}{n}}.
\end{equation}
The first term is the approximation error from representing the density on the level-$L$ dyadic partition (`quantization'), and the second is the stochastic error from estimating the corresponding bin probabilities from $n$ samples. For each fixed truncation level $L$, the quantity $\varepsilon_n(L)$ is the level-$L$ bias--variance scale governing the sup-norm accuracy of a truncated P\'olya tree posterior.

The role of the theory below is to determine which truncation levels are admissible under the two models for the data. In the `noiseless' setting of \citet{castillo2017polya} the coefficients are observed directly, and truncation levels up to $L_{\mathrm{Cast}}$, defined below, are admissible. In our regression setting, by contrast, the empirical bin counts are not observed and must instead be inferred from the noisy linear observation $\bY$. This creates an additional admissibility constraint: the dyadic partition cannot be finer than the resolution at which the latent coefficients, and hence their bin counts, are recoverable from the regression model. The resulting contraction rate is obtained by evaluating $\varepsilon_n(L)$ at the largest truncation level satisfying both the direct-observation admissibility constraint and this additional regression-induced constraint. When the latter constraint is inactive, the rate reduces to the direct-observation rate of \citet{castillo2017polya}.

\subsection{Review: posterior contraction from direct observations}

We first record the direct-observation benchmark corresponding to \citet{castillo2017polya}. Since here $\bfbeta_0=(\beta_{01},\dots,\beta_{0n})$ is observed, the problem reduces to density estimation from i.i.d.~samples from $g_0$. The theorem below adapts Castillo's posterior contraction result for P\'olya trees to our notation and assumptions, including compact support $[a,b]$ and the truncated P\'olya tree prior described in Section~\ref{sec:pt}.

Let
\[
L_{\mathrm{Cast}}
:=
\left\lfloor
\log_2\!\left(
c_{\mathrm{Cast}}\left(\frac{n}{\log n}\right)^{\frac{1}{2\alpha+1}}
\right)
\right\rfloor
\]
for a sufficiently small constant $c_{\mathrm{Cast}}>0$. This is the largest truncation level allowed by the direct-observation argument.
For $L\le L_{\mathrm{Cast}}$, let $\Pi_L(\cdot\mid \bfbeta_0)$ denote the posterior distribution on $g$ under the level-$L$ truncated P\'olya tree prior given the direct observations $\bfbeta_0$.

Define
\[
\Lambda_n(l):=\sqrt{(l+L)\,n\,2^{-l}},
\qquad
P_0(I_k^l):=\int_{I_k^l} g_0(x)\,dx.
\]
For $M>0$, let $\mathcal{B}=\mathcal{B}(M,L)$ be the set of all $\bfbeta\in [a,b]^n$ such that
\[
\big|N_{\bfbeta}(I_k^l)-nP_0(I_k^l)\big|
\le
M \{\Lambda_n(l)\vee(l+L)\}
\]
simultaneously for all $0\le l\le L$ and $0\le k<2^l$. 
Thus, $\mathcal{B}$ is the event that the empirical bin counts are uniformly close to their expectations under $g_0$. 
With this notation, the direct-observation benchmark can be stated as follows. 

\begin{theorem}[adapted from \citet{castillo2017polya}]
\label{thm:beta-observed}
Assume the true model \eqref{eq:prior-iid}, with $g_0$ satisfying the conditions of Section~\ref{sec:model}. For every sufficiently large $M>0$, there exist constants $C_0,C_1,c_0>0$ such that for any sequence $L\le L_{\mathrm{Cast}}$,
\[
\sup_{\bfbeta\in \mathcal{B}(M,L)}
\Pi_L\!\left(
\|g-g_0\|_\infty > C_0\,\varepsilon_n(L)
\;\middle|\;
\bfbeta
\right)
\le C_1e^{-c_0L},
\]
and
\[
\PP_{g_0}\!\left(\bfbeta_0\in \mathcal{B}(M,L)\right)\ge 1-C_1e^{-c_0L}.
\]
Consequently, if in addition $L \to \infty$, then $\varepsilon_n(L)\to0$ and
\[
\EE_{g_0}\!\left[
\Pi_L\!\left(
\|g-g_0\|_\infty > C_0\,\varepsilon_n(L)
\;\middle|\;
\bfbeta_0
\right)
\right]
\longrightarrow 0
\qquad\text{as } n\to\infty.
\]
\end{theorem}

\subsection{Posterior contraction from noisy linear observations}

We now turn to the main technical contribution, which can be viewed as an extension of Castillo's result from the directly-observed coefficients model to the noisy linear model. Relative to the direct-observation benchmark, the new element is an additional restriction on the truncation level, reflecting the accuracy with which the latent coefficients can be recovered from $\bY$. We encode this restriction below through $L_{\mathrm{noise}}$.

\smallskip
Let $\lambda_{\min}(X^\top X)$ denote the smallest eigenvalue of $X^\top X$. For a sufficiently small constant $c_{\mathrm{noise}}>0$, define
\[
L_{\mathrm{noise}}
:=
\max\left\{
\ell\in\mathbb N:
\frac{2^{3\ell/2}}{\ell^{3/2}}
\le
c_{\mathrm{noise}}\,\frac{\lambda_{\min}(X^\top X)}
{\sigma^2 n^{3/2}}
\right\}.
\] 
We choose the truncation level as
\begin{equation}
\label{eq:Lmn-def}
L:=\min\{L_{\mathrm{Cast}},\,L_{\mathrm{noise}}\}.
\end{equation}
Thus $L_{\mathrm{Cast}}$ is the upper bound from \citet{castillo2017polya}, while $L_{\mathrm{noise}}$ is the additional resolution limit due to observing $\bY$ instead of $\bfbeta_0$. 
We can now state the main result of the paper.

\begin{theorem}[Posterior contraction]
\label{thm:main}
Assume the true model \eqref{eq:prior-iid} and \eqref{eq:lik-linear}, with $g_0$ satisfying the conditions of Section~\ref{sec:model}. 
Let $L$ be given by \eqref{eq:Lmn-def}, and suppose that $X$ is such that
\[
L\to\infty
\qquad\text{as } m,n\to\infty.
\]
Then $\varepsilon_n(L)\to0$, and there exists a constant $C_0>0$ such that
\[
\EE_{g_0}\!\left[
\Pi_L\!\left(
\|g-g_0\|_\infty > C_0\,\varepsilon_n(L)
\;\middle|\;
\bY
\right)
\right]
\longrightarrow 0
\qquad\text{as } m,n\to\infty.
\]
\end{theorem}

Thus the contraction rate has the same functional form as in the direct-observation problem, but evaluated at the smaller admissible depth $L=\min\{L_{\mathrm{Cast}},\,L_{\mathrm{noise}}\}$. When $L_{\mathrm{noise}} \ge L_{\mathrm{Cast}}$, the regression problem incurs no additional loss and the Castillo direct-observation rate is recovered.
\paragraph{Posterior mean estimation of the coefficients.}
The preceding theorem concerns recovery of the mixing density $g_0$. As part of its proof, we show and use an auxiliary contraction bound for the latent coefficient vector $\bfbeta$, stated as Proposition~\ref{prop:beta-l2-contract} in Appendix~\ref{app:beta-contraction}. 
Thus, $\bfbeta$ contraction is not derived from the density contraction statement; instead, it relies on the given likelihood function, and is one of the elements used to prove posterior convergence of $\Pi_L$.
In addition to its use in the proof of Theorem~\ref{thm:main}, Proposition~\ref{prop:beta-l2-contract} also yields the following consequence for the Bayes estimator, under squared loss, of the coefficient vector $\bfbeta_0$.
Let
\[
\widehat{\bfbeta}(\bY)
:=
\int \bfbeta\,\Pi_{L,\beta}(d\bfbeta\mid \bY)
\]
denote the posterior mean of the coefficient vector under the working P\'olya tree model.
\begin{corollary}[$\ell_2$ error of the posterior mean]
\label{cor:beta-posterior-mean-l2}
Assume the true model \eqref{eq:prior-iid} and \eqref{eq:lik-linear}, with $g_0$ supported on $[a,b]$ and satisfying $g_0\le M_0$.
Let $L=L(m,n)$ be any sequence of truncation levels, not necessarily the truncation level used in Theorem~\ref{thm:main}.
Suppose $X$ has full column rank and that
\[
\frac{\sigma^2\{n+2^L\log(n+1)\}}{\lambda_{\min}(X^\top X)}
\longrightarrow 0.
\]
Then
\[
\EE_{g_0}\left[
\|\widehat{\bfbeta}(\bY)-\bfbeta_0\|_2^2
\right]
\longrightarrow 0.
\]
\end{corollary}

Corollary~\ref{cor:beta-posterior-mean-l2} concerns a different target from Theorem~\ref{thm:main}. For coefficient recovery in the regime considered here, the regression likelihood drives the estimation of $\bfbeta_0$, while the truncation level $L$ enters through the complexity of the induced marginal prior on $\bfbeta$. 
Thus smaller truncation levels make the displayed sufficient condition easier to satisfy, but this should not be interpreted as a recommendation for the density recovery problem. For instance, when $L=0$, the working prior on $g$ is the uniform density on $[a,b]$; such a prior may still yield consistent estimation of $\bfbeta_0$ under a well-conditioned design, but it cannot consistently estimate a non-uniform mixing density $g_0$.

\begin{remark}[Normal sequence model with fixed $\sigma^2$]
The normal sequence model with fixed noise level $\sigma^2$ can be written as the special case of \eqref{eq:lik-linear} with $m=n$ and $X=I_n$. However, Theorem~\ref{thm:main} does not yield a consistency result in that case. Indeed, when $X=I_n$, $\lambda_{\min}(X^\top X)=1$, so $L_{\mathrm{noise}} \not\to \infty$. Hence the theorem is vacuous in that case.

This is not to say that the P\'olya tree posterior is inconsistent in the normal sequence model, only that the present proof strategy, developed for a regression setting, does not apply to the sequence case. 
The reason is that the argument used here controls the P\'olya tree posterior by first showing that the posterior for the latent coefficient vector $\bfbeta$ concentrates sufficiently tightly around the realized vector $\bfbeta_0$ in $\ell_2$, and then using this to show that replacing $\bfbeta_0$ by a posterior draw $\bfbeta$ does not substantially change the empirical bin counts at the growing resolution $L$. As $L\to\infty$, the cells in the dyadic partition shrink, and so this route requires increasingly fine coordinate-level control of $\bfbeta$ around $\bfbeta_0$. Such control is unavailable in the normal sequence model with fixed noise level, where each coefficient is observed only once and the uncertainty in each coordinate does not vanish. A consistency proof for that setting would therefore require a different argument.
\end{remark}

\subsection{Random designs}

We next state a random-design consequence of Theorem~\ref{thm:main}.
In this subsection, probability statements involving $X$ are with respect to the random-design law, denoted by $\PP_X$.
For a random design, write $L_{\mathrm{noise}}(X)$ for $L_{\mathrm{noise}}$ evaluated at the realized matrix $X$, and set
\[
L(X):=\min\{L_{\mathrm{Cast}},L_{\mathrm{noise}}(X)\}.
\]

The relevant quantity is
\[
\frac{\lambda_{\min}(X^\top X)}{\sigma^2 n^{3/2}},
\]
which controls $L_{\mathrm{noise}}(X)$. For Gaussian random designs with i.i.d.~rows whose second-moment matrices have eigenvalues uniformly bounded away from zero, Lemma~\ref{lem:gaussian-random-design-conditioning} in Appendix~\ref{app:gaussian-random-designs} shows that, provided the design is sufficiently tall,
\[
\lambda_{\min}(X^\top X)\gtrsim m
\]
with probability tending to one. Thus, under the growth condition
\[
\frac{m}{\sigma^2 n^{3/2}}\to\infty,
\]
we obtain $L(X)\to\infty$ in $\PP_X$-probability, which is the random-design analogue of the fixed-design assumption in Theorem~\ref{thm:main}. This yields the following random-design consequence.
\begin{corollary}[Gaussian random designs]
\label{cor:random-design}
Assume the true data-generating model of Section~\ref{sec:model}, except that the design matrix $X\in\RR^{m\times n}$ is random and independent of $\bfbeta_0$. Conditional on $(X,\bfbeta_0)$,
\[
\bY\mid X,\bfbeta_0
\sim
\calN(X\bfbeta_0,\sigma^2 I_m).
\]
The P\'olya tree posterior is computed conditional on the realized design $X$, treating $X$ as fixed.

Suppose the rows $x_1,\dots,x_m\in\RR^n$ of $X$ are i.i.d. Gaussian vectors with
\[
x_i\sim \calN(\bmu_n,\Sigma_n),
\qquad
\Omega_n:=\Sigma_n+\bmu_n\bmu_n^\top,
\qquad
\lambda_{\min}(\Omega_n)\ge \kappa
\]
for some constant $\kappa>0$ independent of $m,n$. Suppose further that
\[
m\ge \Gamma n
\]
for all sufficiently large $m,n$, where $\Gamma$ is the universal constant from Lemma~\ref{lem:gaussian-random-design-conditioning}, and that
\[
\frac{m}{\sigma^2 n^{3/2}}\longrightarrow\infty .
\]
Let $L(X)$ be the random truncation level defined above.
Then $L(X)\to\infty$ in $\PP_X$-probability, $\varepsilon_n(L(X))\to0$ in $\PP_X$-probability, and there exists a constant $C_0>0$ such that
\[
\EE_X\EE_{g_0}^X\!\left[
\Pi_{L(X),X}\!\left(
\|g-g_0\|_\infty>C_0\varepsilon_n(L(X))
\mid
\bY
\right)
\right]
\longrightarrow 0.
\]
Here $\EE_X$ denotes expectation over the random design, $\EE_{g_0}^X$ denotes expectation over $(\bfbeta_0,\bY)$ conditional on the realized design $X$, and $\Pi_{L(X),X}(\cdot\mid\bY)$ denotes the P\'olya tree posterior computed with truncation level $L(X)$ and fixed design $X$.
\end{corollary}

\section{Proof sketch of Theorem~\ref{thm:main}}
\label{sec:proof-sketch}

We now explain the main ideas behind the proof of Theorem~\ref{thm:main}; the full proof appears in Appendix~\ref{app:proof-main}.
At a high level, the proof transfers the direct-observation P\'olya tree benchmark of Theorem~\ref{thm:beta-observed} to the regression setting. This transfer is not automatic: the benchmark applies once the coefficient vector has the correct empirical bin counts, whereas in the regression model the coefficients are latent. The main additional work, and the main technical challenge, is therefore to show that, under suitable conditioning of the design matrix and for truncation levels $L\le L_{\mathrm{noise}}$, the posterior for $\bfbeta$ concentrates tightly enough around the realized vector $\bfbeta_0$ for those count conditions to remain valid. We break the proof outline into the following steps.

\bigskip

\noindent\hspace*{\parindent}\textbf{Step 1. Reduction to the direct-observation benchmark.}\enspace 
Denote the complement of the target sup-norm ball by 
\[
T_g:=\{g:\|g-g_0\|_\infty > C_0\,\varepsilon_n(L)\}. 
\]
Using the posterior mixture representation from Section~\ref{sec:pt},
\[
\Pi_L(T_g\mid \bY)
=
\int \Pi_L(T_g\mid \bfbeta)\,\Pi_{L,\beta}(d\bfbeta\mid \bY),
\]
we can separate the contribution from ``good'' coefficient vectors and ``bad'' ones. 
That is, for any set $\mathcal{B}\subset [a,b]^n$,
\begin{align}
\Pi_L(T_g\mid \bY)
&=
\int_{\mathcal{B}} \Pi_L(T_g\mid \bfbeta)\,\Pi_{L,\beta}(d\bfbeta\mid \bY)
+
\int_{\mathcal{B}^c} \Pi_L(T_g\mid \bfbeta)\,\Pi_{L,\beta}(d\bfbeta\mid \bY) \nonumber\\
&\le
\sup_{\bfbeta\in\mathcal{B}} \Pi_L(T_g\mid \bfbeta)
+
\Pi_{L,\beta}(\bfbeta\notin\mathcal{B}\mid \bY).
\label{eq:proof-decomp-pointwise}
\end{align}
Taking expectation under $\PP_{g_0}$ yields
\begin{equation}
\label{eq:tower}
\EE_{g_0}\!\left[\Pi_L(T_g\mid \bY)\right]
\le
\sup_{\bfbeta\in\mathcal{B}} \Pi_L(T_g\mid \bfbeta)
+
\EE_{g_0}\!\left[\Pi_{L,\beta}(\bfbeta\notin\mathcal{B}\mid \bY)\right].
\end{equation}
This is formalized in Appendix~\ref{app:proof-main} as Lemma~\ref{lem:appB-tower}.

We choose $\mathcal{B}=\mathcal{B}(M,L)$ to be the good-count event from Theorem~\ref{thm:beta-observed}, namely the set of vectors whose empirical bin counts are uniformly close to the corresponding probabilities under $g_0$. 
For this choice, the first term in \eqref{eq:tower} is controlled directly by the direct-observation benchmark Theorem~\ref{thm:beta-observed}:
\[
\sup_{\bfbeta\in\mathcal{B}(M,L)} \Pi_L(T_g\mid \bfbeta)\le C_1e^{-c_0L}.
\]
Therefore, the proof of Theorem~\ref{thm:main} reduces to showing that the posterior for the latent coefficient vector $\bfbeta$ places asymptotically negligible mass outside $\mathcal{B}(M,L)$.

\bigskip
\noindent\hspace*{\parindent}\textbf{Step 2. From bin-count control to \texorpdfstring{$\ell_2$}{l2}-control of \texorpdfstring{$\bfbeta$}{beta}.}\enspace 
We next explain why $\bfbeta$ membership in $\mathcal{B}(M,L)$ can be enforced by showing that $\bfbeta$ is sufficiently close to the true coefficient vector $\bfbeta_0$ in $\ell_2$.

For each cell $I_k^l$, by the triangle inequality,
\begin{align}
\big|N_{\bfbeta}(I_k^l)-nP_0(I_k^l)\big|
&\le
\big|N_{\bfbeta}(I_k^l)-N_{\bfbeta_0}(I_k^l)\big|
+
\big|N_{\bfbeta_0}(I_k^l)-nP_0(I_k^l)\big|.
\label{eq:count-triangle}
\end{align}
The second term is uniformly small over all cells whenever $\bfbeta_0\in\mathcal{B}(M,L)$, by definition of $\mathcal{B}(M,L)$. This event occurs with high $\PP_{g_0}$-probability, by Theorem~\ref{thm:beta-observed}.

The new issue in the regression model is the first term in \eqref{eq:count-triangle}, which measures how much the bin counts change when $\bfbeta_0$ is replaced by a posterior draw $\bfbeta$.
A count can change only through coordinates $j$ that cross a cell boundary when moving from $\beta_{0j}$ to $\beta_j$.
The key observation is that such a crossing can occur only if either $\beta_{0j}$ and $\beta_j$ are sufficiently far apart, or if $\beta_{0j}$ is close to a cell boundary, such that a small separation between $\beta_{0j}$ and $\beta_j$ can still cross a cell boundary.

Lemma~\ref{lem:appB-bincount-stability} formalizes this observation by bounding a cell's count discrepancy in terms of two quantities: the number of true coefficients lying near the relevant cell boundary, and the $\ell_2$-distance between $\bfbeta$ and $\bfbeta_0$.
The first quantity is controlled uniformly over all cells and levels by Lemma~\ref{lem:appB-boundary-counts}, with high $\PP_{g_0}$-probability. Together, these two lemmas show that posterior draws $\bfbeta$ have the required empirical bin counts whenever they are sufficiently close to $\bfbeta_0$ in $\ell_2$. More precisely, the proof reduces the remaining task to showing that the posterior probability of
\[
\|\bfbeta-\bfbeta_0\|_2^2
\gtrsim
\frac{L^{3/2}}{n^{1/2}2^{3L/2}}
\]
vanishes. Thus, after the count-stability argument, it remains to prove posterior $\ell_2$-contraction of the latent coefficient vector $\bfbeta$.

\bigskip
\noindent\hspace*{\parindent}\textbf{Step 3. Posterior contraction of the latent coefficients.}\enspace 
The required $\ell_2$-contraction is provided by Proposition~\ref{prop:beta-l2-contract}. It shows that, for suitable constants
$A,c>0$,
\[
\EE_{g_0}\!\left[
\Pi_{L,\beta}\!\left(
\|\bfbeta-\bfbeta_0\|_2
>
A\sigma
\sqrt{
\frac{n+2^L\log(n+1)}
{\lambda_{\min}(X^\top X)}
}
\;\middle|\;
\bY
\right)
\right]
\le 2e^{-cn}.
\]
The factor $\lambda_{\min}(X^\top X)$ quantifies how well the design separates different coefficient vectors: better conditioning of $X$ yields stronger recovery of $\bfbeta_0$.

Since $L\le L_{\mathrm{Cast}}$, the term $2^L\log(n+1)$ is $O(n)$, so the posterior $\ell_2$-radius is essentially
\[
\sigma\sqrt{\frac{n}{\lambda_{\min}(X^\top X)}}.
\]
The definition of $L_{\mathrm{noise}}$ ensures that this radius is small enough relative to the count-stability threshold from Step 2. Indeed, $L\le L_{\mathrm{noise}}$ implies
\[
\frac{\sigma^2 n}{\lambda_{\min}(X^\top X)}
\lesssim
\frac{L^{3/2}}{n^{1/2}2^{3L/2}}.
\]
Thus, under the truncation choice in Theorem~\ref{thm:main}, posterior draws of $\bfbeta$ are close enough to $\bfbeta_0$ to preserve the empirical bin count conditions required by the direct-observation benchmark.

\bigskip

\noindent\hspace*{\parindent}\textbf{Step 4. Conclusion.}\enspace 
Combining the count-stability argument from Step 2 with the posterior $\ell_2$-contraction from Step 3 gives
\[
\EE_{g_0}\!\left[
\Pi_{L,\beta}(\bfbeta\notin\mathcal B(M,L)\mid\bY)
\right]\longrightarrow 0.
\]
Together with the direct-observation bound from Step 1 and the decomposition \eqref{eq:tower}, this yields
\[
\EE_{g_0}\!\left[
\Pi_L(T_g\mid\bY)
\right]\longrightarrow 0.
\]
Since $T_g=\{g:\|g-g_0\|_\infty>C_0\varepsilon_n(L)\}$, this is exactly the conclusion of Theorem~\ref{thm:main}.

\section*{Acknowledgments}
A.W. was supported by the Israeli Science Foundation (ISF) under grant no.~2679/24.


\bibliographystyle{abbrvnat}
\bibliography{references}


\appendix

\section{Proof of Theorem~\ref{thm:beta-observed}}

This appendix proves the direct-observation benchmark, Theorem~\ref{thm:beta-observed}. It follows the argument of \citet{castillo2017polya}, adapted to our setup and notation (support $[a,b]$, truncated P\'olya tree, $\mathrm{Beta}(1,1)$ at all levels). Throughout this appendix, set
\[
B:=b-a,
\]
and write $C,c>0$ for finite positive constants whose values may change from line to line. We will also repeatedly use the consequence
of the choice $L\le L_{\mathrm{Cast}}$, that
\[
L2^L/n=o(1).
\]

Because this appendix almost identically follows the structure of \citet{castillo2017polya}, we at times omit certain computational details that are unchanged from that work, and refer the reader to \citet{castillo2017polya} for these details.

\subsection*{A.1. Haar basis on $[a,b]$}

Let $T:[a,b]\to[0,1]$ be the affine map $T(x)=(x-a)/B$. Let
\[
\phi=\mathbbm{1}_{[0,1]},
\qquad
\psi=-\mathbbm{1}_{[0,1/2)}+\mathbbm{1}_{[1/2,1]}
\]
on $[0,1]$. Define the rescaled Haar functions on $[a,b]$ by
\[
\varphi(x):=B^{-1/2}\mathbbm{1}_{[a,b]}(x),
\qquad
\psi_{lk}(x):=B^{-1/2}2^{l/2}\psi\{2^lT(x)-k\},
\quad l\ge 0,\quad 0\le k<2^l .
\]
Then
\[
\{\varphi\}\cup\{\psi_{lk}:l\ge 0,\ 0\le k<2^l\}
\]
is an orthonormal basis of $L^2([a,b])$. For a square-integrable function $f$, write
\[
f_{lk}:=\langle f,\psi_{lk}\rangle_2,
\qquad
f_\varphi:=\langle f,\varphi\rangle_2 .
\]
Since $g_0$ is $\alpha$-H\"older on $[a,b]$, with $\alpha\in(0,1]$, the Haar coefficients satisfy
\begin{equation}
\label{eq:app-haar-holder-bound}
\sup_{l\ge 0}\sup_{0\le k<2^l}
\left(\frac{2^l}{B}\right)^{1/2+\alpha}|g_{0,lk}|<\infty ,
\end{equation}
which is the analogue on $[a,b]$ of the corresponding bound on $[0,1]$ recalled in \citet{castillo2017polya}.

For an integer $L\ge 0$, and a function $f$ in $L^2([a,b])$, let $f^{(L)}$ denote the $L^2$-projection of $f$ onto the linear span of $\{\varphi\}\cup\{\psi_{lk}:0\le l\le L-1,\ 0\le k<2^l\}$. Denote $f^{(L^c)}:=f-f^{(L)} .$

This convention matches the depth-$L$ truncated P\'olya tree: a draw $g$ from the prior is constant on the terminal cells $I_k^L$ and hence satisfies $g=g^{(L)}$.

\subsection*{A.2. Good bin-count set}

For $0\le l\le L$ and $0\le k<2^l$, recall the notation
\[
N_{\bfbeta}(I_k^l):=\sum_{j=1}^n\mathbbm{1}\{\beta_j\in I_k^l\},
\qquad
P_0(I_k^l):=\int_{I_k^l}g_0(x)\,dx, 
\]
\[
\Lambda_n(l):=\sqrt{(l+L)n2^{-l}}.
\]
For $M>0$, define
\begin{equation}
\label{eq:app-B-def}
\mathcal B(M,L)
:=
\left\{
\bfbeta\in[a,b]^n:
\left|N_{\bfbeta}(I_k^l)-nP_0(I_k^l)\right|
\le
M\{\Lambda_n(l)\vee(l+L)\}
\ \text{for all }0\le l\le L,\ 0\le k<2^l
\right\}.
\end{equation}

\begin{lemma}[Bin-count concentration]
\label{lem:app-bin-count-concentration}
For $M$ sufficiently large, there exist $C,c>0$ such that
\[
\PP_{g_0}\{\bfbeta_0\in\mathcal B(M,L)\}\ge 1-Ce^{-cL}.
\]
\end{lemma}

\begin{proof}
Fix $l,k$ and set $p_{lk}:=P_0(I_k^l)$. By Bernstein's inequality, for all $t>0$,
\[
\PP_{g_0}\left(
\left|N_{\bfbeta_0}(I_k^l)-np_{lk}\right|>t
\right)
\le
2\exp\left\{-\frac{t^2/2}{np_{lk}(1-p_{lk})+t/3}\right\}.
\]
Since $g_0$ is bounded above and below on $[a,b]$, $p_{lk}\asymp 2^{-l}$ uniformly over $l,k$. Taking
\[
t=M\{\Lambda_n(l)\vee(l+L)\}
\]
and considering separately the regimes $\Lambda_n(l)\ge l+L$ and $\Lambda_n(l)<l+L$, we obtain
\[
\PP_{g_0}\left(
\left|N_{\bfbeta_0}(I_k^l)-np_{lk}\right|>
M\{\Lambda_n(l)\vee(l+L)\}
\right)
\le
C e^{-cM(l+L)} .
\]
A union bound over $k=0,\ldots,2^l-1$ and $l=0,\ldots,L$ gives
\[
\PP_{g_0}\{\bfbeta_0\notin\mathcal B(M,L)\}
\le
\sum_{l=0}^L 2^l C e^{-cM(l+L)}
\le Ce^{-cL},
\]
provided $M$ is chosen sufficiently large.
\end{proof}

\subsection*{A.3. Auxiliary bounds}

Let $\Pi_L(\cdot\mid\bfbeta)$ denote the P\'olya tree posterior under direct observations $\bfbeta$. Let
\[
\bar g:=\EE_{\Pi_L}[g\mid\bfbeta]
\]
be the posterior mean density, and define
\[
\bar P(I_k^l):=\int_{I_k^l}\bar g(x)\,dx,
\qquad
P_0(I_k^l):=\int_{I_k^l}g_0(x)\,dx .
\]
For a parent cell $I_k^l$, write
\[
\bar Y_{l,k}:= \frac{\bar P(I_{2k}^{l+1})}{\bar P(I_k^l)},
\qquad
y_{l,k}:=\frac{P_0(I_{2k}^{l+1})}{P_0(I_k^l)} .
\]
The complementary right-child split proportions are $1-\bar Y_{l,k}$ and $1-y_{l,k}$.

\begin{lemma}
\label{lem:app-mass-ratio}
For every $0\le l\le L$ and $0\le k<2^l$, on $\bfbeta\in\mathcal B(M,L)$,
\[
\left|\frac{\bar P(I_k^l)}{P_0(I_k^l)}-1\right|
\le
C\left(\frac{2^l}{n}+\sqrt{\frac{L2^l}{n}}\right).
\]
\end{lemma}

\begin{proof}
For fixed cell $I_k^l$, define its level-$i$ ancestor by
\[
s_i=s_i(l,k):=\left\lfloor \frac{k}{2^{l-i}}\right\rfloor,
\qquad
I_i^{(l,k)}:=I_{s_i}^{i},
\qquad 0\le i\le l .
\]
Thus $I_0^{(l,k)}=I_0^0$ and $I_l^{(l,k)}=I_k^l$. Along the path from the root to $I_k^l$,
\[
\bar P(I_k^l)=\prod_{i=1}^l \bar q_i^{(l,k)},
\qquad
P_0(I_k^l)=\prod_{i=1}^l q_i^{(l,k)},
\]
where
\[
q_i^{(l,k)}:=\frac{P_0(I_i^{(l,k)})}{P_0(I_{i-1}^{(l,k)})}
\]
is the true split proportion from the level-$(i-1)$ ancestor into the level-$i$ ancestor, and $\bar q_i^{(l,k)}$ is the corresponding posterior mean split proportion. Under the Beta$(1,1)$ splits,
\[
\bar q_i^{(l,k)}
=
\frac{1+N_{\bfbeta}(I_i^{(l,k)})}
{2+N_{\bfbeta}(I_{i-1}^{(l,k)})}.
\]
On $\bfbeta \in \mathcal B(M,L)$, write
\[
N_{\bfbeta}(I_i^{(l,k)})=nP_0(I_i^{(l,k)})+\delta_i^{(l,k)},
\qquad
|\delta_i^{(l,k)}|\le M\{\Lambda_n(i)\vee(i+L)\}.
\]

This yields
\[
\bar q_i^{(l,k)} = q_i^{(l,k)} \frac{1 + n^{-1}(1 + \delta_i^{(l,k)})/P_0(I_i^{(l,k)})}{1 + n^{-1}(2 + \delta_{i-1}^{(l,k)})/P_0(I_{i-1}^{(l,k)})}
\]
Using $|\delta_i^{(l,k)}| \lesssim \sqrt{nL2^{-i}} \ll n2^{-i} \lesssim nP_0(I_i^{(l,k)})$ (the second to last inequality because $L2^{L} = o(n)$, the last one because $g_0$ is bounded away from $0$), we deduce the denominator is bounded away from $0$, and so

\[
\left|\frac{\bar q_i^{(l,k)}}{q_i^{(l,k)}}-1\right|
\le
C\left(\frac{2^i}{n}+\sqrt{\frac{L2^i}{n}}\right).
\]
Summing this bound over $i\le l$ gives
\[
\sum_{i=1}^l\left|\frac{\bar q_i^{(l,k)}}{q_i^{(l,k)}}-1\right|
\le
C\left(\frac{2^l}{n}+\sqrt{\frac{L2^l}{n}}\right),
\]
and the elementary product bound in Lemma~\ref{lem:app-product-bound} yields the claim.
\end{proof}

\begin{lemma}
\label{lem:app-split-ratio}
For every $0\le l\le L-1$ and $0\le k<2^l$, on $\bfbeta\in\mathcal B(M,L)$,
\[
\left|\frac{\bar Y_{l,k}}{y_{l,k}}-1\right|
\le
C\frac{2^{l/2}}{n}
\left(2^lB^{1/2}|g_{0,lk}|+\sqrt{nL}\right).
\]
The same bound holds for the right-child split ratios after replacing $\bar Y_{l,k},y_{l,k}$ by $1-\bar Y_{l,k},1-y_{l,k}$.
\end{lemma}

\begin{proof}
The posterior mean split into the left child is
\[
\bar Y_{l,k}
=
\frac{1+N_{\bfbeta}(I_{2k}^{l+1})}{2+N_{\bfbeta}(I_k^l)}.
\]
Expanding this expression around
\[
y_{l,k}=\frac{P_0(I_{2k}^{l+1})}{P_0(I_k^l)}
\]
similarly to the proof of Lemma~\ref{lem:app-mass-ratio} gives
\[
\left|\frac{\bar Y_{l,k}}{y_{l,k}}-1\right|
\le
\left | \frac{1 + n^{-1}(1 + \delta_{l+1,2k})/P_0(I_{2k}^{l+1})}{1 + n^{-1}(2 + \delta_{l,k})/P_0(I_k^l)} - 1\right |,
\]
where
\[
\delta_{r,s}:=N_{\bfbeta}(I_s^r)-nP_0(I_s^r).
\]
Using the bin-count bounds from $\mathcal B(M,L)$ to conclude the denominator is bounded away from $0$, similarly to the proof of Lemma~\ref{lem:app-mass-ratio}, yields
\[
\left|\frac{\bar Y_{l,k}}{y_{l,k}}-1\right|
\le
\frac{C}{n}\left|P_0(I_{2k}^{l+1})^{-1}-2P_0(I_k^l)^{-1}\right|
+
\frac{C}{n}\left(
\frac{|\delta_{l+1,2k}|}{P_0(I_{2k}^{l+1})}
+
\frac{|\delta_{l,k}|}{P_0(I_k^l)}
\right).
\]

The first term is controlled by the Haar coefficient identity
\[
|1-2y_{l,k}|
=
\frac{B^{1/2}2^{-l/2}|g_{0,lk}|}{P_0(I_k^l)}
\]
and by $P_0(I_k^l)\asymp P_0(I_{2k}^{l+1})\asymp 2^{-l}$. Thus it is bounded by
\[
C\frac{2^{3l/2}}{n}B^{1/2}|g_{0,lk}|.
\]
The second term is bounded by
\[
C\frac{2^l}{n}\sqrt{nL2^{-l}}
=
C\sqrt{\frac{L2^l}{n}},
\]
using the bin-count bounds from $\mathcal B(M,L)$. Combining the two bounds gives the stated inequality. The right-child statement follows identically, since the right-child split is one minus the left-child split and both true child probabilities are uniformly bounded away from zero and one.
\end{proof}

\begin{lemma}[Product bound]
\label{lem:app-product-bound}
Let $\{y_i\}_{i=1}^J$ and $\{w_i\}_{i=1}^J$ be positive sequences. If
\[
\max_{1\le i\le J}\left|\frac{w_i}{y_i}-1\right|\le c_1<1,
\qquad
\sum_{i=1}^J\left|\frac{w_i}{y_i}-1\right|\le c_2<\infty,
\]
then there exists a constant $c_3$ depending on $c_1,c_2$ only such that
\[
\left|\prod_{i=1}^J\frac{w_i}{y_i}-1\right|
\le
c_3\sum_{i=1}^J\left|\frac{w_i}{y_i}-1\right|.
\]
\end{lemma}

\begin{proof}
The statement of this lemma is unchanged from \citet{castillo2017polya}. Hence we do not repeat it here, and refer the reader to \citet{castillo2017polya} for the details of the proof.
\end{proof}

\begin{lemma}[Beta tail bound]
\label{lem:app-beta-tail}
Let $Z\sim\mathrm{Beta}(\varphi,\psi)$ with $\varphi,\psi>0$. Suppose that, for some constants $0<c_0<c_1<1$,
\[
c_0\le \frac{\varphi}{\varphi+\psi}\le c_1,
\qquad
\varphi\wedge\psi>8.
\]
Then there exists $D>0$, depending on $c_0,c_1$ only such that for any $x>0$,
\[
\mathbb P\left(
|Z-\EE Z|>
\frac{x}{\sqrt{\varphi+\psi}}+\frac{2}{\varphi+\psi}
\right)
\le
D e^{-x^2/4}.
\]
\end{lemma}

\begin{proof}
Again the statement of this lemma is identical to \citet{castillo2017polya}, and so we do not repeat the proof here.
\end{proof}

\subsection*{A.4. Proof of Theorem~\ref{thm:beta-observed}}

\begin{proof}
Fix $\bfbeta\in\mathcal B(M,L)$, and work under the posterior probability measure $\Pi_L(\cdot\mid\bfbeta)$. Let $g$ denote the generic posterior random density, and let
\[
\bar g:=\EE_{\Pi_L}[g\mid\bfbeta]
\]
denote its posterior mean.

Since the P\'olya tree is truncated at depth $L$, every posterior draw $g$ and the posterior mean $\bar g$ are constant on level-$L$ dyadic cells.
Thus, for every posterior draw $g$, $g=g^{(L)}$ and $\bar g=\bar g^{(L)}$. Decompose
\begin{equation}
\label{eq:app-basic-decomp}
\|g-g_0\|_\infty
\le
\|g-\bar g\|_\infty
+
\|\bar g-g_0^{(L)}\|_\infty
+
\|g_0^{(L^c)}\|_\infty .
\end{equation}

\paragraph{Bias term $\|g_0^{(L^c)}\|_\infty$.}
Using the Haar expansion, the bound
\[
\left\|\sum_{k=0}^{2^l-1}|\psi_{lk}|\right\|_\infty
\le
B^{-1/2}2^{l/2},
\]
and \eqref{eq:app-haar-holder-bound},
\begin{equation}
\label{eq:app-bias-bound}
\begin{aligned}
\|g_0^{(L^c)}\|_\infty
&\le
\sum_{l=L}^{\infty}
\left(\max_{0\le k<2^l}|g_{0,lk}|\right)
\left\|\sum_{k=0}^{2^l-1}|\psi_{lk}|\right\|_\infty \\
&\le
C\sum_{l=L}^{\infty}
\left(\frac{B}{2^l}\right)^{1/2+\alpha}B^{-1/2}2^{l/2}
\le
C B^\alpha 2^{-\alpha L}.
\end{aligned}
\end{equation}

\paragraph{Posterior centering term $\|\bar g-g_0^{(L)}\|_\infty$.}
For $0\le l\le L-1$, the Haar coefficients of $\bar g$ and $g_0$ satisfy
\[
\bar g_{lk}
=
\sqrt{\frac{2^l}{B}}\bar P(I_k^l)(1-2\bar Y_{l,k}),
\qquad
 g_{0,lk}
=
\sqrt{\frac{2^l}{B}}P_0(I_k^l)(1-2y_{l,k}).
\]
Hence
\[
\bar g_{lk}-g_{0,lk}
=
g_{0,lk}\left\{\frac{\bar P(I_k^l)}{P_0(I_k^l)}-1\right\}
+2\sqrt{\frac{2^l}{B}}\bar P(I_k^l)(y_{l,k}-\bar Y_{l,k}).
\]
By Lemmas~\ref{lem:app-mass-ratio} and \ref{lem:app-split-ratio}, together with $P_0(I_k^l)\asymp2^{-l}$ and $\bar P(I_k^l)\asymp P_0(I_k^l)$ on $\bfbeta \in \mathcal B(M,L)$,
\[
|\bar g_{lk}-g_{0,lk}|
\le
C|g_{0,lk}|\left(\frac{2^l}{n}+\sqrt{\frac{L2^l}{n}}\right)
+C\sqrt{\frac{L}{nB}} .
\]
Therefore
\[
\begin{aligned}
\|\bar g-g_0^{(L)}\|_\infty
&\le
\sum_{l=0}^{L-1}
\left(\max_{0\le k<2^l}|\bar g_{lk}-g_{0,lk}|\right)
\left\|\sum_{k=0}^{2^l-1}|\psi_{lk}|\right\|_\infty \\
&\le
C B^{-1}\sqrt{\frac{L2^L}{n}}
+
C\sum_{l=0}^{L-1}B^\alpha2^{-\alpha l}
\left(\frac{2^l}{n}+\sqrt{\frac{L2^l}{n}}\right),
\end{aligned}
\]
where we used \eqref{eq:app-haar-holder-bound} to bound $|g_{0,lk}|$. Under $L\le L_{\mathrm{Cast}}$, the last sum is bounded by
\[
C\left(B^\alpha2^{-\alpha L}+B^{-1}\sqrt{\frac{L2^L}{n}}\right),
\]
after adjusting constants. Thus
\begin{equation}
\label{eq:app-centering-bound}
\|\bar g-g_0^{(L)}\|_\infty
\le
C\varepsilon_n(L).
\end{equation}

\paragraph{Posterior stochastic term $\|g-\bar g\|_\infty$.}
Continue to fix $\bfbeta\in\mathcal B(M,L)$. Under the posterior $\Pi_L(\cdot\mid\bfbeta)$, let $g$ denote a generic posterior draw, let $\widetilde P$ denote the probability measure induced by $g$, and define the corresponding split variables by
\[
\widetilde Y_{l,k}
:=
\frac{\widetilde P(I_{2k}^{l+1})}{\widetilde P(I_k^l)},
\qquad 0\le l\le L-1,\quad 0\le k<2^l .
\]
Define the event $\mathcal A=\mathcal A(M,L)$ by
\[
\mathcal A
:=
\left\{
\left|\widetilde Y_{l,k}-\bar Y_{l,k}\right|
\le
M\sqrt{\frac{L}{nP_0(I_{2k}^{l+1})}}
\ \text{for all }0\le l\le L-1,\ 0\le k<2^l
\right\}.
\]
Conditional on $\bfbeta$, the split $\widetilde Y_{l,k}$ has distribution
\[
\mathrm{Beta}\{1+N_{\bfbeta}(I_{2k}^{l+1}),\ 1+N_{\bfbeta}(I_{2k+1}^{l+1})\}.
\]
On $\bfbeta \in \mathcal B(M,L)$, both Beta parameters are at least $8$ for all large $n$, and the corresponding posterior mean is bounded away from $0$ and $1$ by Lemma~\ref{lem:app-split-ratio}, uniformly over $l,k$. Lemma~\ref{lem:app-beta-tail}, followed by a union bound over the $O(2^L)$ split variables, gives
\begin{equation}
\label{eq:app-A-prob}
\sup_{\bfbeta\in\mathcal B(M,L)}
\Pi_L(\mathcal A^c\mid\bfbeta)
\le
Ce^{-cL}.
\end{equation}

It remains to show that, for every fixed $\bfbeta\in\mathcal B(M,L)$, the event $\mathcal A$ implies the desired bound on $\|g-\bar g\|_\infty$.
For $0\le l\le L-1$,
\[
\begin{aligned}
g_{lk}-\bar g_{lk}
&=
\sqrt{\frac{2^l}{B}}\,\bar P(I_k^l)
\left[
\frac{\widetilde P(I_k^l)}{\bar P(I_k^l)}
(1-2\widetilde Y_{l,k})
-
(1-2\bar Y_{l,k})
\right] \\
&=
2\sqrt{\frac{2^l}{B}}\,\bar P(I_k^l)
(\bar Y_{l,k}-\widetilde Y_{l,k}) \\
&\quad+
\sqrt{\frac{2^l}{B}}\,\bar P(I_k^l)
\left[
\frac{\widetilde P(I_k^l)}{\bar P(I_k^l)}-1
\right]
\left[
1-2\bar Y_{l,k}
+
2(\bar Y_{l,k}-\widetilde Y_{l,k})
\right] \\
&=
2\sqrt{\frac{2^l}{B}}\,\bar P(I_k^l)
(\bar Y_{l,k}-\widetilde Y_{l,k}) \\
&\quad+
\left[
\frac{\widetilde P(I_k^l)}{\bar P(I_k^l)}-1
\right]
\left[
\bar g_{lk}
+
2\sqrt{\frac{2^l}{B}}\,\bar P(I_k^l)
(\bar Y_{l,k}-\widetilde Y_{l,k})
\right].
\end{aligned}
\]

It remains to control the random mass ratio appearing in the preceding display. Fix $0\le l\le L-1$ and $0\le k<2^l$, and consider the path from the root to $I_k^l$. Along this path,
\[
\frac{\widetilde P(I_k^l)}{\bar P(I_k^l)}
=
\prod_{i=1}^l
\frac{\widetilde q_i^{(l,k)}}{\bar q_i^{(l,k)}},
\]
where $\widetilde q_i^{(l,k)}$ and $\bar q_i^{(l,k)}$ denote the corresponding random and posterior-mean split proportions, either to the left child or to the right child.

As before, on $\bfbeta\in\mathcal B(M,L)$, by Lemma~\ref{lem:app-split-ratio}, we have that the $\bar q_i^{(l,k)}$ are uniformly bounded away from zero and one. Hence, on $\mathcal A$ for $\bfbeta \in \mathcal B(M,L)$,
\[
\left|
\frac{\widetilde q_i^{(l,k)}}{\bar q_i^{(l,k)}}-1
\right|
\le
C\left|\widetilde q_i^{(l,k)}-\bar q_i^{(l,k)}\right|
\le
C\sqrt{\frac{L2^i}{n}} .
\]
Since
\[
\sum_{i=1}^l\sqrt{\frac{L2^i}{n}}
\le
C\sqrt{\frac{L2^l}{n}}
=o(1)
\]
uniformly for $l\le L$, Lemma~\ref{lem:app-product-bound} gives
\begin{equation}
\label{eq:app-random-mass-ratio}
\left|
\frac{\widetilde P(I_k^l)}{\bar P(I_k^l)}-1
\right|
\le
C\sqrt{\frac{L2^l}{n}} .
\end{equation}

Additionally, by the definition of $\mathcal A$ and the fact that $P_0(I_{2k}^{l+1})\asymp 2^{-l}$,
\[
|\widetilde Y_{l,k}-\bar Y_{l,k}|
\le
C\sqrt{\frac{L2^l}{n}} .
\]
Combining this with \eqref{eq:app-random-mass-ratio} in the preceding coefficient decomposition, and using
$\bar P(I_k^l)\asymp P_0(I_k^l)\asymp 2^{-l}$ on $\mathcal B(M,L)$, yields
\[
|g_{lk}-\bar g_{lk}|
\le
C|\bar g_{lk}|\sqrt{\frac{L2^l}{n}}
+
C\sqrt{\frac{L}{nB}} .
\]
By the triangle inequality, $|\bar g_{lk}| \leq |\bar g_{lk} - g_{0,lk}| + |g_{0,lk}|$. We have controlled both of those terms in the above centering term step. Plugging in those bounds, and repeating the same algebra as in the above step, yields
\begin{equation}
\label{eq:app-stochastic-bound}
\|g-\bar g\|_\infty
\le
C\left(
B^\alpha2^{-\alpha L}
+
B^{-1}\sqrt{\frac{L2^L}{n}}
\right)
=
C\varepsilon_n(L).
\end{equation}

\paragraph{Conclusion.}
Combining \eqref{eq:app-basic-decomp}, \eqref{eq:app-bias-bound}, \eqref{eq:app-centering-bound}, and \eqref{eq:app-stochastic-bound}, we obtain that, on $\mathcal A$ for $\bfbeta \in \mathcal B(M,L)$,
\[
\|g-g_0\|_\infty
\le
C_0\left(B^\alpha2^{-\alpha L}+B^{-1}\sqrt{\frac{L2^L}{n}}\right)
=
C_0\varepsilon_n(L).
\]
This bound and \eqref{eq:app-A-prob} imply
\[
\begin{aligned}
&\sup_{\bfbeta\in\mathcal B(M,L)}
\Pi_L\left(
\|g-g_0\|_\infty>C_0\varepsilon_n(L)
\mid\bfbeta
\right) \\
&\qquad\le
\sup_{\bfbeta\in\mathcal B(M,L)}
\Pi_L\left(
\mathcal A\cap\{\|g-g_0\|_\infty>C_0\varepsilon_n(L)\}
\mid\bfbeta
\right)
+
\sup_{\bfbeta\in\mathcal B(M,L)}
\Pi_L(\mathcal A^c\mid\bfbeta) \\
&\qquad\le C_1e^{-c_0L}.
\end{aligned}
\]
Together with Lemma~\ref{lem:app-bin-count-concentration}, this proves the two displayed bounds in Theorem~\ref{thm:beta-observed}. If additionally $L\to\infty$, then
\[
B^\alpha 2^{-\alpha L}\to0,
\qquad
B^{-1}\sqrt{\frac{L2^L}{n}}\to0,
\]
where the second convergence follows from $L\le L_{\mathrm{Cast}}$. Hence $\varepsilon_n(L)\to0$.

Finally,
\[
\begin{aligned}
\EE_{g_0}\left[
\Pi_L\left(
\|g-g_0\|_\infty>C_0\varepsilon_n(L)
\mid\bfbeta_0
\right)
\right]
&\le
\sup_{\bfbeta\in\mathcal B(M,L)}
\Pi_L\left(
\|g-g_0\|_\infty>C_0\varepsilon_n(L)
\mid\bfbeta
\right) \\
&\quad+
\PP_{g_0}\{\bfbeta_0\notin\mathcal B(M,L)\}
\le 2C_1e^{-c_0L},
\end{aligned}
\]
which tends to zero whenever $L\to\infty$. This completes the proof.
\end{proof}

\section{Proof of Theorem~\ref{thm:main}}
\label{app:proof-main}

We prove Theorem~\ref{thm:main}. Throughout this appendix, write $C,c>0$ for finite positive constants whose values may change from line to line.

Fix $C_0>0$, chosen large below, and set
\[
T_g := \{g:\|g-g_0\|_\infty > C_0\,\varepsilon_n(L)\}.
\]
As in the proof of Theorem~\ref{thm:beta-observed}, because $L\to\infty$ and $L\le L_{\mathrm{Cast}}$,
\[
B^\alpha 2^{-\alpha L}\to0,
\qquad
B^{-1}\sqrt{\frac{L2^L}{n}}\to0.
\]
Hence $\varepsilon_n(L)\to0$.

For $0\le l\le L$, write
\[
R_l:=\Lambda_n(l)\vee(l+L).
\]
For $M>0$, define
\[
\mathcal B(M,L)
:=
\left\{
\bfbeta\in[a,b]^n:
\big|N_{\bfbeta}(I_k^l)-nP_0(I_k^l)\big|
\le
M R_l,
\quad
0\le l\le L,\ 0\le k<2^l
\right\}.
\]
Choose $M>0$ large enough that Theorem~\ref{thm:beta-observed} applies with both $M$ and $M/2$. Define
\[
\mathcal B_0:=\mathcal B(M/2,L),
\qquad
\mathcal B_1:=\mathcal B(M,L).
\]

By Theorem~\ref{thm:beta-observed}, there exist $C,c>0$ such that
\begin{equation}
\label{eq:appB-beta0-in-B0}
\PP_{g_0}\!\left(\bfbeta_0\in\mathcal B_0\right)
\ge 1-Ce^{-cL},
\end{equation}
and, after increasing $C_0$ if necessary,
\begin{equation}
\label{eq:appB-direct-observation-on-B1}
\sup_{\bfbeta\in\mathcal B_1}
\Pi_L(T_g\mid \bfbeta)
\le Ce^{-cL}.
\end{equation}

\subsection*{B.1 Tower decomposition}

\begin{lemma}[Tower decomposition]
\label{lem:appB-tower}
Under the working hierarchical P\'olya tree model,
\[
\EE_{g_0}\!\left[\Pi_L(T_g\mid \bY)\right]
\le
\sup_{\bfbeta\in\mathcal B_1}
\Pi_L(T_g\mid \bfbeta)
+
\EE_{g_0}\!\left[
\Pi_{L,\beta}(\bfbeta\notin\mathcal B_1\mid \bY)
\right].
\]
\end{lemma}

\begin{proof}
By the posterior mixture representation and the conditional independence $g\perp \bY\mid \bfbeta$ under the working model,
\[
\Pi_L(T_g\mid \bY)
=
\int \Pi_L(T_g\mid \bfbeta)\,\Pi_{L,\beta}(d\bfbeta\mid \bY).
\]
Splitting the integral over $\mathcal B_1$ and $\mathcal B_1^c$, taking the supremum over $\mathcal B_1$, and using $\Pi_L(T_g\mid\bfbeta)\le 1$ on $\mathcal B_1^c$, gives
\[
\Pi_L(T_g\mid \bY)
\le
\sup_{\bfbeta\in\mathcal B_1}\Pi_L(T_g\mid \bfbeta)
+
\Pi_{L,\beta}(\bfbeta\notin\mathcal B_1\mid \bY).
\]
Taking $\EE_{g_0}$ proves the claim.
\end{proof}

Combining Lemma~\ref{lem:appB-tower} with
\eqref{eq:appB-direct-observation-on-B1}, we obtain
\[
\EE_{g_0}\!\left[\Pi_L(T_g\mid \bY)\right]
\le
Ce^{-cL}
+
\EE_{g_0}\!\left[
\Pi_{L,\beta}(\bfbeta\notin\mathcal B_1\mid \bY)
\right].
\]
Thus it remains to prove
\begin{equation}
\label{eq:appB-goal-beta-B1}
\EE_{g_0}\!\left[
\Pi_{L,\beta}(\bfbeta\notin\mathcal B_1\mid \bY)
\right]
\to 0.
\end{equation}

\subsection*{B.2 Reducing failure of \texorpdfstring{$\mathcal B_1$}{B1} to bin-count instability}

We work first on the event $\{\bfbeta_0\in\mathcal B_0\}$, whose complement has $\PP_{g_0}$-probability at most $Ce^{-cL}$ by \eqref{eq:appB-beta0-in-B0}. On this event, uniformly over $0\le l\le L$ and $0\le k<2^l$,
\[
\big|N_{\bfbeta_0}(I_k^l)-nP_0(I_k^l)\big|
\le \frac{M}{2} R_l.
\]
If, in addition, $\bfbeta\notin\mathcal B_1$, then for some $l,k$,
\[
\big|N_{\bfbeta}(I_k^l)-nP_0(I_k^l)\big|>M R_l.
\]
Therefore, by the triangle inequality,
\[
\big|N_{\bfbeta}(I_k^l)-N_{\bfbeta_0}(I_k^l)\big|
>
\frac{M}{2} R_l.
\]
Define
\[
U
:=
\bigcup_{l=0}^L\bigcup_{k<2^l}
\left\{
\big|N_{\bfbeta}(I_k^l)-N_{\bfbeta_0}(I_k^l)\big|>\frac{M}{2}R_l
\right\}.
\]
Then
\[
\Pi_{L,\beta}(\bfbeta\notin\mathcal B_1\mid \bY)
\le
\mathbbm 1_{\{\bfbeta_0\notin\mathcal B_0\}}
+
\Pi_{L,\beta}(U\mid \bY).
\]
Taking $\EE_{g_0}$ and using \eqref{eq:appB-beta0-in-B0}, we obtain
\begin{equation}
\label{eq:appB-reduce-to-U}
\EE_{g_0}\!\left[
\Pi_{L,\beta}(\bfbeta\notin\mathcal B_1\mid \bY)
\right]
\le
Ce^{-cL}
+
\EE_{g_0}\!\left[
\Pi_{L,\beta}(U\mid \bY)
\right].
\end{equation}

\subsection*{B.3 A deterministic bin-count stability lemma}

For an interval $I\subset[a,b]$ and $\tau>0$, define
\[
I^{+\tau}:=\{x\in[a,b]:\mathrm{dist}(x,I)\le \tau\},
\qquad
I^{-\tau}:=\{x\in I:\mathrm{dist}(x,I^c)\ge \tau\}.
\]
Thus $I^{+\tau}\setminus I^{-\tau}$ is the $\tau$-neighborhood of the boundary of $I$, restricted to $[a,b]$.

\begin{lemma}[Bin-count stability under $\ell_2$ perturbations]
\label{lem:appB-bincount-stability}
For any interval $I\subset[a,b]$, any $\tau>0$, and any $u,v\in[a,b]^n$,
\[
|N_u(I)-N_v(I)|
\le
N_v(I^{+\tau}\setminus I^{-\tau})
+
\frac{\|u-v\|_2^2}{\tau^2}.
\]
\end{lemma}

\begin{proof}
If $v_j\in I^{-\tau}$ and $|u_j-v_j|\le \tau$, then $u_j\in I$.
Similarly, if $v_j\notin I^{+\tau}$ and $|u_j-v_j|\le \tau$, then $u_j\notin I$.
Hence a mismatch
\[
\mathbbm 1\{u_j\in I\}\ne \mathbbm 1\{v_j\in I\}
\]
can occur only if either
\[
v_j\in I^{+\tau}\setminus I^{-\tau}
\qquad\text{or}\qquad
|u_j-v_j|>\tau.
\]
Therefore,
\[
|N_u(I)-N_v(I)|
\le
N_v(I^{+\tau}\setminus I^{-\tau})
+
\#\{j:|u_j-v_j|>\tau\}.
\]
Finally,
\[
\#\{j:|u_j-v_j|>\tau\}
\le
\frac{\|u-v\|_2^2}{\tau^2},
\]
which proves the lemma.
\end{proof}

Applying Lemma~\ref{lem:appB-bincount-stability} with $u=\bfbeta$, $v=\bfbeta_0$, and $I=I_k^l$, we obtain, for any choice of $\tau_l>0$,
\begin{equation}
\label{eq:appB-bin-stab-applied}
\big|N_{\bfbeta}(I_k^l)-N_{\bfbeta_0}(I_k^l)\big|
\le
N_{\bfbeta_0}\!\left((I_k^l)^{+\tau_l}\setminus (I_k^l)^{-\tau_l}\right)
+
\frac{\|\bfbeta-\bfbeta_0\|_2^2}{\tau_l^2}.
\end{equation}

\subsection*{B.4 Controlling the boundary-strip counts}

We now choose
\begin{equation}
\label{eq:appB-tau-choice}
\tau_l:=c_\tau\frac{R_l}{n},
\qquad 0\le l\le L,
\end{equation}
where $c_\tau>0$ is a sufficiently small constant.

Define the good boundary-count event
\[
E_{\partial}
:=
\left\{
N_{\bfbeta_0}\!\left((I_k^l)^{+\tau_l}\setminus (I_k^l)^{-\tau_l}\right)
\le
\frac{M}{4} R_l
\quad
\text{for all }0\le l\le L,\ 0\le k<2^l
\right\}.
\]

\begin{lemma}[Uniform boundary-count control]
\label{lem:appB-boundary-counts}
For $M>0$ sufficiently large and $c_\tau>0$ sufficiently small, there exists $C,c>0$ such that
\[
\PP_{g_0}(E_{\partial}^c)\le Ce^{-cL}.
\]
\end{lemma}

\begin{proof}
Fix $0\le l\le L$ and $0\le k<2^l$. The boundary strip
\[
(I_k^l)^{+\tau_l}\setminus (I_k^l)^{-\tau_l}
\]
is contained in the union of at most two intervals of total length at most $4\tau_l$. Since $g_0$ is bounded above, there exists a constant $M_0<\infty$ such that
\[
p_{lk}
:=
\PP_{g_0}\!\left(
\beta_{01}\in (I_k^l)^{+\tau_l}\setminus (I_k^l)^{-\tau_l}
\right)
\le M_0\tau_l.
\]
Thus
\[
Z_{lk}
:=
N_{\bfbeta_0}\!\left((I_k^l)^{+\tau_l}\setminus (I_k^l)^{-\tau_l}\right)
\sim \mathrm{Bin}(n,p_{lk}),
\]
with
\[
\EE_{g_0}Z_{lk}=np_{lk}\le M_0n\tau_l=M_0c_\tau R_l.
\]

Choose $c_\tau < \frac{M}{8M_0}$, so that $\frac{M}{4}R_l  -\mathbb E Z_{lk}\ge (M/8)R_l > 0$. Then by Bernstein's inequality on $Z_{lk} - \mathbb E Z_{lk}$, and the bound $p_{lk}\le M_0\tau_l$,
\[
\PP_{g_0}\!\left(
Z_{lk}>\frac{M}{4}R_l
\right)
\le
C\exp\left\{
-c_{\tau,1}
\frac{R_l^2}{n\tau_l+R_l}
\right\}
\]
for a constant $c_{\tau,1}>0$ which increases with $M$. Since $n\tau_l=c_\tau R_l$, we obtain
\[
\PP_{g_0}\!\left(
Z_{lk}>\frac{M}{4}R_l
\right)
\le
C\exp(-c_{\tau,2}R_l)
\]
for a constant $c_{\tau,2}>0$ which increases with $M$. Because $R_l\ge l+L$, this further implies
\[
\PP_{g_0}\!\left(
Z_{lk}>\frac{M}{4}R_l
\right)
\le
C\exp\{-c_{\tau,2}(l+L)\}.
\]

Taking a union bound over all levels and cells,
\[
\begin{aligned}
\PP_{g_0}(E_{\partial}^c)
&\le
\sum_{l=0}^L\sum_{k<2^l}
C\exp\{-c_{\tau,2}(l+L)\} \\
&=
Ce^{-c_{\tau,2}L}
\sum_{l=0}^L
\exp\{-(c_{\tau,2}-\log 2)l\}.
\end{aligned}
\]
Choosing $M>0$ sufficiently large so that $c_{\tau,2}>\log 2$, the final sum is bounded by a constant. Therefore, for some $C,c>0$,
\[
\PP_{g_0}(E_{\partial}^c)\le Ce^{-cL}.
\]
\end{proof}

\subsection*{B.5 Reduction to an \texorpdfstring{$\ell_2$}{l2}-posterior tail}

We now control the term $\Pi_{L,\beta}(U\mid \bY)$ appearing in \eqref{eq:appB-reduce-to-U}. On the event $E_{\partial}$, if $U$ occurs, then for some $0\le l\le L$ and $0\le k<2^l$,
\[
\big|N_{\bfbeta}(I_k^l)-N_{\bfbeta_0}(I_k^l)\big|>\frac{M}{2}R_l.
\]
By \eqref{eq:appB-bin-stab-applied} and the definition of $E_{\partial}$,
\[
\frac{M}{2}R_l
<
\frac{M}{4}R_l
+
\frac{\|\bfbeta-\bfbeta_0\|_2^2}{\tau_l^2}.
\]
Therefore
\[
\|\bfbeta-\bfbeta_0\|_2^2
>
\frac{M}{4}\tau_l^2R_l.
\]
Define
\begin{equation}
\label{eq:appB-aL-def}
a_L
:=
\frac{M}{4}
\min_{0\le l\le L}
\left\{
\tau_l^2R_l
\right\}.
\end{equation}
Then
\[
E_{\partial}\cap U
\subseteq
\left\{
\|\bfbeta-\bfbeta_0\|_2^2>a_L
\right\}.
\]
Hence,
\[
\Pi_{L,\beta}(U\mid \bY)
\le
\mathbbm 1_{E_{\partial}^c}
+
\Pi_{L,\beta}\!\left(
\|\bfbeta-\bfbeta_0\|_2^2>a_L
\;\middle|\; \bY
\right).
\]
Taking $\EE_{g_0}$ and applying Lemma~\ref{lem:appB-boundary-counts}, we obtain
\[
\EE_{g_0}\!\left[
\Pi_{L,\beta}(U\mid \bY)
\right]
\le
Ce^{-cL}
+
\EE_{g_0}\!\left[
\Pi_{L,\beta}\!\left(
\|\bfbeta-\bfbeta_0\|_2^2>a_L
\;\middle|\; \bY
\right)
\right].
\]
Combining this with \eqref{eq:appB-reduce-to-U},
\begin{equation}
\label{eq:appB-beta-not-B1-bound}
\EE_{g_0}\!\left[
\Pi_{L,\beta}(\bfbeta\notin\mathcal B_1\mid \bY)
\right]
\le
Ce^{-cL}
+
\EE_{g_0}\!\left[
\Pi_{L,\beta}\!\left(
\|\bfbeta-\bfbeta_0\|_2^2>a_L
\;\middle|\; \bY
\right)
\right].
\end{equation}

It remains to lower bound $a_L$. By \eqref{eq:appB-tau-choice},
\[
\tau_l^2R_l
=
c_\tau^2\frac{R_l^3}{n^2}.
\]
Thus
\[
a_L
\asymp
\min_{0\le l\le L}\frac{R_l^3}{n^2}.
\]
Since $R_l\ge \Lambda_n(l)$ and
\[
\Lambda_n(l)
=
\sqrt{(l+L)n2^{-l}},
\]
we have, for all $0\le l\le L$,
\[
\Lambda_n(l)
\gtrsim
\sqrt{Ln2^{-L}}.
\]
Therefore
\[
R_l^3
\ge
\Lambda_n(l)^3
\gtrsim
(Ln2^{-L})^{3/2}.
\]
It follows that
\begin{equation}
\label{eq:appB-aL-lower}
a_L
\gtrsim
\frac{(Ln2^{-L})^{3/2}}{n^2}
=
\frac{L^{3/2}}{n^{1/2}2^{3L/2}}.
\end{equation}

\subsection*{B.6 Latent-coefficient contraction and conclusion}

We now invoke the posterior contraction result for the latent coefficient vector.
Namely, by Proposition~\ref{prop:beta-l2-contract}, for a sufficiently large constant $A > 0$,
\[
\EE_{g_0}\!\left[
\Pi_{L,\beta}\!\left(
\|\bfbeta-\bfbeta_0\|_2
>
A \sigma
\sqrt{\frac{n + 2^L\log(n+1)}{\lambda_{\min}(X^\top X)}}
\;\middle|\; \bY
\right)
\right]
\to 0.
\]

Because $L \leq L_{\mathrm{Cast}}$, $2^L \log(n+1) = O(n)$. Thus, for $A$ sufficiently large,
\[
\EE_{g_0}\!\left[
\Pi_{L,\beta}\!\left(
\|\bfbeta-\bfbeta_0\|_2
>
A \sigma
\sqrt{\frac{n}{\lambda_{\min}(X^\top X)}}
\;\middle|\; \bY
\right)
\right]
\to 0.
\]

By the definition of $L_{\mathrm{noise}}$ and the fact that $L\le L_{\mathrm{noise}}$,
\[
\frac{2^{3L/2}}{L^{3/2}}
\le
c_{\mathrm{noise}}\,\frac{\lambda_{\min}(X^\top X)}{\sigma^2 n^{3/2}}.
\]
Equivalently,
\[
\frac{\sigma^2 n}{\lambda_{\min}(X^\top X)}
\le
c_{\mathrm{noise}}\,\frac{L^{3/2}}{n^{1/2}2^{3L/2}}.
\]
Combining this with \eqref{eq:appB-aL-lower}, and choosing $c_{\mathrm{noise}}>0$ sufficiently small relative to the constants in \eqref{eq:appB-aL-lower} and Proposition~\ref{prop:beta-l2-contract}, we obtain
\[
A^2\sigma^2\frac{n}{\lambda_{\min}(X^\top X)}
\le
a_L.
\]
Therefore,
\[
\left\{
\|\bfbeta-\bfbeta_0\|_2^2>a_L
\right\}
\subseteq
\left\{
\|\bfbeta-\bfbeta_0\|_2
>
A\sigma\sqrt{\frac{n}{\lambda_{\min}(X^\top X)}}
\right\}.
\]
Using Proposition~\ref{prop:beta-l2-contract} in \eqref{eq:appB-beta-not-B1-bound} gives
\[
\EE_{g_0}\!\left[
\Pi_{L,\beta}(\bfbeta\notin\mathcal B_1\mid \bY)
\right]
\longrightarrow 0.
\]
Together with the decomposition from Lemma~\ref{lem:appB-tower} and the bound
\[
\sup_{\bfbeta\in\mathcal B_1}\Pi_L(T_g\mid \bfbeta)\le Ce^{-cL},
\]
this yields
\[
\EE_{g_0}\!\left[
\Pi_L(T_g\mid \bY)
\right]
\longrightarrow 0,
\]
which proves Theorem~\ref{thm:main}.
\qed

\section{Posterior contraction of \texorpdfstring{$\bfbeta$}{beta}}
\label{app:beta-contraction}

This appendix proves the latent-coefficient contraction bound used in Appendix~\ref{app:proof-main}, and derives Corollary~\ref{cor:beta-posterior-mean-l2} on the $\ell_2$ error of the posterior mean.

Recall that $\Pi_{L,\beta}$ denotes the marginal prior distribution of $\bfbeta=(\beta_1,\dots,\beta_n)$ induced by the level-$L$ truncated P\'olya tree prior on $g$, after integrating out $g$.
Let $\pi_{L,\beta}$ denote its Lebesgue density on $[a,b]^n$,
\[
\pi_{L,\beta}(\bfbeta)
=
\int \prod_{j=1}^n g(\beta_j)\,\Pi_L(dg),
\qquad
\bfbeta=(\beta_1,\dots,\beta_n)\in [a,b]^n.
\]

\begin{proposition}[$\bfbeta$ posterior contraction]
\label{prop:beta-l2-contract}
Assume the true model \eqref{eq:prior-iid} and \eqref{eq:lik-linear}, with $g_0$ supported on $[a,b]$ and satisfying $g_0\le M_0$.
Let $L=L(m,n)$ be any sequence of truncation levels, not necessarily the truncation level used in Theorem~\ref{thm:main}. Suppose $X$ has full column rank. Then there exist finite constants $A,c>0$ such that
\begin{equation}
\label{eq:appC-rnL-contraction}
\EE_{g_0}\!\left[
\Pi_{L,\beta}\!\left(
\left\{
\bfbeta:
\|\bfbeta-\bfbeta_0\|_2
>
A\sigma
\sqrt{\frac{n+2^L\log(n+1)}{\lambda_{\min}(X^\top X)}}
\right\}
\;\middle|\;
\bY
\right)
\right]
\le 2e^{-c n}
\longrightarrow 0.
\end{equation}
\end{proposition}

\begin{proof}
We argue in three steps.
First, we derive a uniform lower bound for the marginal prior density $\pi_{L,\beta}$.
Second, we prove a prediction-norm contraction bound when the working P\'olya tree model is used both to generate and analyze the data.
Third, a change-of-measure argument transfers this bound to the true data-generating law $\PP_{g_0}$, and the prediction-norm bound is converted into an $\ell_2$-bound using the conditioning of $X$.

\medskip
\noindent\textbf{Step 1. Lower bound for the marginal P\'olya tree prior on $\bfbeta$.}
Let $B:=b-a$. Recall $V_{l,k}$ denotes the random split proportion assigned to the left child of $I_k^l$, so that $1-V_{l,k}$ is assigned to the right child.
Under the truncated P\'olya tree prior considered here, the variables $\{V_{l,k}:0\le l\le L-1,\ 0\le k<2^l\}$ are independent $\mathrm{Beta}(1,1)$ random variables.
Conditional on these split variables, the density at a point $\beta_j$ is the terminal-bin density $2^L/B$ times the product of the split proportions along the path from the root to the terminal bin containing $\beta_j$.
Hence, for $\bfbeta=(\beta_1,\dots,\beta_n)$,
\[
\prod_{j=1}^n g(\beta_j)
=
\left(\frac{2^L}{B}\right)^n
\prod_{l=0}^{L-1}\prod_{k=0}^{2^l-1}
V_{l,k}^{N_{\bfbeta}^{(0)}(I_k^l)}
(1-V_{l,k})^{N_{\bfbeta}^{(1)}(I_k^l)}.
\]
For a fixed split variable $V_{l,k}$, integrating its factor with respect to its $\mathrm{Beta}(1,1)$ density, i.e. the uniform density on $[0,1]$, gives
\[
\int_0^1
v^{N_{\bfbeta}^{(0)}(I_k^l)}
(1-v)^{N_{\bfbeta}^{(1)}(I_k^l)}
\,dv
=
\mathrm B\!\left(
N_{\bfbeta}^{(0)}(I_k^l)+1,
N_{\bfbeta}^{(1)}(I_k^l)+1
\right).
\]
Multiplying these factors over all internal nodes yields
\[
\pi_{L,\beta}(\bfbeta)
=
\left(\frac{2^L}{B}\right)^n
\prod_{l=0}^{L-1}\prod_{k=0}^{2^l-1}
\mathrm B\!\left(
N_{\bfbeta}^{(0)}(I_k^l)+1,
N_{\bfbeta}^{(1)}(I_k^l)+1
\right),
\]
where $\mathrm B$ denotes the beta function.

Each coordinate of $\bfbeta$ contributes to exactly one node at each level $l=0,\dots,L-1$. Hence
\[
2^{Ln}
=
\prod_{l=0}^{L-1}\prod_{k=0}^{2^l-1}
2^{N_{\bfbeta}(I_k^l)}.
\]
It follows that
\[
\pi_{L,\beta}(\bfbeta)
=
B^{-n}
\prod_{l=0}^{L-1}\prod_{k=0}^{2^l-1}
R\!\left(
N_{\bfbeta}^{(0)}(I_k^l),
N_{\bfbeta}^{(1)}(I_k^l)
\right),
\]
where, for nonnegative integers $u,v$,
\[
R(u,v):=2^{u+v}\mathrm B(u+1,v+1).
\]
If $s=u+v$, then
\[
R(u,v)
=
\frac{2^s}{(s+1)\binom{s}{u}}
\ge
\frac{1}{s+1},
\]
since $\binom{s}{u}\le 2^s$. Applying this inequality with
$s=N_{\bfbeta}(I_k^l)$ yields
\[
\pi_{L,\beta}(\bfbeta)
\ge
B^{-n}
\prod_{l=0}^{L-1}\prod_{k=0}^{2^l-1}
\frac{1}{N_{\bfbeta}(I_k^l)+1}.
\]
Because $N_{\bfbeta}(I_k^l)\le n$,
\[
\prod_{l=0}^{L-1}\prod_{k=0}^{2^l-1}
\{N_{\bfbeta}(I_k^l)+1\}
\le
(n+1)^{\sum_{l=0}^{L-1}2^l}
\le
(n+1)^{2^L}.
\]
Thus, uniformly over $\bfbeta\in[a,b]^n$,
\begin{equation}
\label{eq:appC-prior-density-lower}
\pi_{L,\beta}(\bfbeta)
\ge
B^{-n}(n+1)^{-2^L}.
\end{equation}

Let
\[
q_0(\bfbeta):=\prod_{j=1}^n g_0(\beta_j)
\]
be the density of $\bfbeta_0$ under the true law. Since $g_0\le M_0$, the bound \eqref{eq:appC-prior-density-lower} implies
\[
\frac{q_0(\bfbeta)}{\pi_{L,\beta}(\bfbeta)}
\le
(B M_0)^n(n+1)^{2^L}.
\]
Consequently, for a finite constant $D$,
\begin{equation}
\label{eq:appC-density-ratio}
\frac{q_0(\bfbeta)}{\pi_{L,\beta}(\bfbeta)}
\le
\exp\{D s_{n,L}\},
\qquad
\bfbeta\in[a,b]^n,
\end{equation}
where
\[
s_{n,L}:=n+2^L\log(n+1).
\]

\medskip
\noindent\textbf{Step 2. Contraction under the working data-generating law.}
Let $(\widetilde{\bfbeta}, \widetilde{\bY})$ be jointly distributed according to the working model,
\[
\widetilde{\bfbeta}\sim \Pi_{L,\beta},
\qquad
\widetilde{\bY}\mid \widetilde{\bfbeta}
\sim
\calN(X\widetilde{\bfbeta},\sigma^2I_m),
\]
and denote the corresponding joint law by $\PP_L$. We show that, for every $t>0$,
\begin{equation}
\label{eq:appC-polya-data-gen-tail}
\EE_L\!\left[
\Pi_{L,\beta}\!\left(
\left\{
\bfbeta:
\|X(\bfbeta-\widetilde{\bfbeta})\|_2>\sigma t
\right\}
\;\middle|\;
\widetilde{\bY}
\right)
\right]
\le
2\PP\!\left(\chi_n^2>\frac{t^2}{4}\right),
\end{equation}
where $\EE_L$ denotes expectation over $(\widetilde{\bfbeta}, \widetilde{\bY})$ with respect to $\PP_L$.
To prove \eqref{eq:appC-polya-data-gen-tail}, draw $\bfbeta^\star$ from $\Pi_{L,\beta}(\cdot\mid\widetilde{\bY})$, conditionally independently of $\widetilde{\bfbeta}$ given $\widetilde{\bY}$.
Conditional on $\widetilde{\bY}$, the two vectors $\bfbeta^\star$ and $\widetilde{\bfbeta}$ are independent draws from the same posterior distribution. Therefore,
\[
\EE_L\!\left[
\Pi_{L,\beta}\!\left(
\left\{
\bfbeta:
\|X(\bfbeta-\widetilde{\bfbeta})\|_2>\sigma t
\right\}
\;\middle|\;
\widetilde{\bY}
\right)
\right]
=
\PP_L\!\left(
\|X(\bfbeta^\star-\widetilde{\bfbeta})\|_2>\sigma t
\right).
\]
Let
\[
\widehat{\bfbeta}:=(X^\top X)^{-1}X^\top\widetilde{\bY}
\]
be the ordinary least squares estimator. By the triangle inequality,
\begin{align*}
\PP_L\!\left(
\|X(\bfbeta^\star-\widetilde{\bfbeta})\|_2>\sigma t
\right)
&\le
\PP_L\!\left(
\|X(\bfbeta^\star-\widehat{\bfbeta})\|_2>\frac{\sigma t}{2}
\right) \\
&\quad+
\PP_L\!\left(
\|X(\widehat{\bfbeta}-\widetilde{\bfbeta})\|_2>\frac{\sigma t}{2}
\right).
\end{align*}
Conditional on $\widetilde{\bY}$, $\bfbeta^\star$ and $\widetilde{\bfbeta}$ have the same distribution, while $\widehat{\bfbeta}$ is a function of $\widetilde{\bY}$.
The two probabilities on the right-hand side are therefore equal, and so
\[
\PP_L\!\left(
\|X(\bfbeta^\star-\widetilde{\bfbeta})\|_2>\sigma t
\right)
\le
2\PP_L\!\left(
\|X(\widehat{\bfbeta}-\widetilde{\bfbeta})\|_2>\frac{\sigma t}{2}
\right).
\]
Under $\PP_L$, we can write
\[
\widetilde{\bY}=X\widetilde{\bfbeta}+\sigma Z,
\qquad
Z\sim\calN(0,I_m),
\]
with $Z$ independent of $\widetilde{\bfbeta}$. Thus
\[
X(\widehat{\bfbeta}-\widetilde{\bfbeta})
=
\sigma P_XZ,
\qquad
P_X:=X(X^\top X)^{-1}X^\top.
\]
Since $X$ has full column rank, $P_X$ is the orthogonal projection onto $\operatorname{col}(X)$, which has dimension $n$. Therefore,
\[
\sigma^{-2}\|X(\widehat{\bfbeta}-\widetilde{\bfbeta})\|_2^2
=
\|P_XZ\|_2^2
\sim
\chi_n^2.
\]
This proves \eqref{eq:appC-polya-data-gen-tail}.

\medskip
\noindent\textbf{Step 3. Transfer to the true law $\PP_{g_0}$.}
For $t > 0$, define the posterior tail function
\[
T_t(\boldsymbol u,\by)
:=
\Pi_{L,\beta}\!\left(
\left\{
\bfbeta:
\|X(\bfbeta -\boldsymbol u)\|_2>\sigma t
\right\}
\;\middle|\;
\by
\right).
\]
Let $p_{\boldsymbol u}(\by)$ be the density of $\calN(X\boldsymbol u,\sigma^2I_m)$ evaluated at $\by$.
Under $\PP_{g_0}$, $\bfbeta_0$ has density $q_0$ and $\bY\mid\bfbeta_0=\boldsymbol u$ has density $p_{\boldsymbol u}$. Hence
\begin{align}
\EE_{g_0}T_t(\bfbeta_0,\bY)
&=
\int_{[a,b]^n}\int_{\RR^m}
T_t(\boldsymbol u,\by)
q_0(\boldsymbol u)p_{\boldsymbol u}(\by)
\,d\by\,d\boldsymbol u \nonumber\\
&=
\EE_L\!\left[
\frac{q_0(\widetilde{\bfbeta})}{\pi_{L,\beta}(\widetilde{\bfbeta})}
T_t(\widetilde{\bfbeta},\widetilde{\bY})
\right].
\label{eq:appC-change-of-measure}
\end{align}
Combining \eqref{eq:appC-change-of-measure} with the density-ratio bound \eqref{eq:appC-density-ratio} and the working-model tail bound \eqref{eq:appC-polya-data-gen-tail} gives
\begin{equation}
\label{eq:appC-true-prediction-tail}
\EE_{g_0}\!\left[
\Pi_{L,\beta}\!\left(
\left\{
\bfbeta:
\|X(\bfbeta-\bfbeta_0)\|_2>\sigma t
\right\}
\;\middle|\;
\bY
\right)
\right]
\le
2\exp\{Ds_{n,L}\}
\PP\!\left(\chi_n^2>\frac{t^2}{4}\right).
\end{equation}

Set $t=A\sqrt{s_{n,L}}$, with $A>0$ to be chosen.
If $S\sim\chi_n^2$, then $\EE e^{S/4}=2^{n/2}$. Markov's inequality gives, for every $x>0$,
\[
\PP(S>x)\le e^{-x/4}2^{n/2}.
\]
Taking $x=A^2s_{n,L}/4$ and using $s_{n,L}\ge n$,
\[
\PP\!\left(S>\frac{A^2s_{n,L}}{4}\right)
\le
\exp\!\left\{
-\frac{A^2}{16}s_{n,L}
+\frac{\log 2}{2}n
\right\}
\le
\exp\!\left\{
-\left(\frac{A^2}{16}-\frac{\log 2}{2}\right)s_{n,L}
\right\}.
\]
Therefore \eqref{eq:appC-true-prediction-tail} implies
\[
\EE_{g_0}\!\left[
\Pi_{L,\beta}\!\left(
\left\{
\bfbeta:
\|X(\bfbeta-\bfbeta_0)\|_2>A\sigma\sqrt{s_{n,L}}
\right\}
\;\middle|\;
\bY
\right)
\right]
\le
2\exp\!\left\{
-\left(
\frac{A^2}{16}-\frac{\log 2}{2}-D
\right)s_{n,L}
\right\}.
\]
Choose $A$ large enough that
\[
c := \frac{A^2}{16} - D - \frac{\log 2}{2} > 0.
\]
The last display then implies
\[
\EE_{g_0}\!\left[
\Pi_{L,\beta}\!\left(
\left\{
\bfbeta:
\|X(\bfbeta-\bfbeta_0)\|_2>A\sigma\sqrt{s_{n,L}}
\right\}
\;\middle|\;
\bY
\right)
\right]
\le
2\exp\!\left\{
-c s_{n,L}
\right\}
\le
2\exp\!\left\{
-c n
\right\},
\]
since $s_{n,L}\ge n$.

It remains only to convert this prediction-norm bound into an $\ell_2$-bound. For any $\bfbeta\in\RR^n$,
\[
\|X(\bfbeta-\bfbeta_0)\|_2^2
=
(\bfbeta-\bfbeta_0)^\top X^\top X(\bfbeta-\bfbeta_0)
\ge
\lambda_{\min}(X^\top X)\|\bfbeta-\bfbeta_0\|_2^2.
\]
Thus
\[
\left\{
\bfbeta:
\|\bfbeta-\bfbeta_0\|_2
>
A\sigma\sqrt{\frac{s_{n,L}}{\lambda_{\min}(X^\top X)}}
\right\}
\subseteq
\left\{
\bfbeta:
\|X(\bfbeta-\bfbeta_0)\|_2>A\sigma\sqrt{s_{n,L}}
\right\}.
\]
This proves \eqref{eq:appC-rnL-contraction}.

\end{proof}

\begin{proof}[Proof of Corollary~\ref{cor:beta-posterior-mean-l2}]
Set
\[
S_n
:=
A\sigma
\sqrt{\frac{n+2^L\log(n+1)}{\lambda_{\min}(X^\top X)}} .
\]
By Jensen's inequality,
\[
\|\widehat{\bfbeta}(\bY)-\bfbeta_0\|_2^2
\le
\int
\|\bfbeta-\bfbeta_0\|_2^2
\,\Pi_{L,\beta}(d\bfbeta\mid\bY).
\]
Since both $\bfbeta$ and $\bfbeta_0$ lie in $[a,b]^n$, the integrand is bounded above by $nB^2$. Therefore,
\[
\|\widehat{\bfbeta}(\bY)-\bfbeta_0\|_2^2
\le
S_n^2
+
nB^2
\Pi_{L,\beta}\!\left(
\|\bfbeta-\bfbeta_0\|_2>S_n
\mid \bY
\right).
\]
Taking $\EE_{g_0}$ and applying Proposition~\ref{prop:beta-l2-contract} gives
\[
\EE_{g_0}\left[
\|\widehat{\bfbeta}(\bY)-\bfbeta_0\|_2^2
\right]
\le
A^2\sigma^2
\frac{n+2^L\log(n+1)}{\lambda_{\min}(X^\top X)}
+
2nB^2e^{-cn},
\]
which tends to zero by assumption.
\end{proof}

\section{Proof of Corollary~\ref{cor:random-design}}
\label{app:gaussian-random-designs}

This appendix proves Corollary~\ref{cor:random-design}.
We first prove a lemma that lower bounds $\lambda_{\min}(X^\top X)$ with high probability under the random-design assumptions.
We then use this lemma to reduce the random-design corollary to the fixed-design result of Theorem~\ref{thm:main} on high-probability events under the design law.

\begin{lemma}[Smallest eigenvalue of Gaussian random designs]
\label{lem:gaussian-random-design-conditioning}
Let $X=X_{m,n}\in\RR^{m\times n}$ have independent rows
$x_1,\dots,x_m\in\RR^n$ with
\[
x_i\sim \calN(\bmu_n,\Sigma_n),
\qquad
\Omega_n:=\Sigma_n+\bmu_n\bmu_n^\top .
\]
Assume that
\[
\lambda_{\min}(\Omega_n)\ge \kappa
\]
for some constant $\kappa>0$ independent of $m,n$.
Then there exist universal constants $\Gamma,\xi>0$ such that whenever $m\ge \Gamma n$,
\[
\PP_X\!\left(
\operatorname{rank}(X)=n
\quad\text{and}\quad
\lambda_{\min}(X^\top X)\ge \frac{\kappa m}{4}
\right)
\ge
1-2\exp(-\xi m).
\]
\end{lemma}

\begin{proof}
Let
\[
Z:=X\Omega_n^{-1/2}.
\]
The rows $z_i:=\Omega_n^{-1/2}x_i$ of $Z$ are independent Gaussian random vectors.
Additionally,
\[
\EE[z_i z_i^\top]
=
\Omega_n^{-1/2}\EE[x_i x_i^\top]\Omega_n^{-1/2}
=
\Omega_n^{-1/2}\Omega_n\Omega_n^{-1/2}
=
I_n.
\]
Thus the rows of $Z$ are isotropic in second moment.

We next check that the sub-Gaussian norms of the rows are uniformly bounded.
For every $u\in S^{n-1}$, the random variable $\langle z_i,u\rangle$ is Gaussian. Write
\[
\langle z_i,u\rangle \sim \calN(a_u,\tau_u^2).
\]
Since $\EE[z_i z_i^\top]=I_n$,
\[
a_u^2+\tau_u^2
=
\EE\langle z_i,u\rangle^2
=
1.
\]
Thus $|a_u|\le1$ and $\tau_u\le1$. If $G\sim\calN(0,1)$, then
\[
\langle z_i,u\rangle \stackrel{d}{=} a_u+\tau_u G.
\]
Therefore, by Minkowski's inequality and the Gaussian moment bound $(\EE|G|^p)^{1/p}\le C_0\sqrt p$, $p\ge1$,
\[
\left(\EE|\langle z_i,u\rangle|^p\right)^{1/p}
\le
|a_u|+\tau_u\left(\EE|G|^p\right)^{1/p}
\le
1+C_0\sqrt p
\le
C_1\sqrt p,
\qquad p\ge1,
\]
where $C_0,C_1>0$ are universal constants. Hence
\[
\|\langle z_i,u\rangle\|_{\psi_2}
=
\sup_{p\ge1}
p^{-1/2}
\left(\EE|\langle z_i,u\rangle|^p\right)^{1/p}
\le C_1,
\]
uniformly over $u\in S^{n-1}$. Consequently,
\[
\|z_i\|_{\psi_2}
= \sup_{u\in S^{n-1}}\|\langle z_i,u\rangle\|_{\psi_2}
\le C_1.
\]

We may now apply \citet[Theorem~5.39]{vershynin2012introduction} to the matrix $Z$.
Since the rows of $Z$ are independent, isotropic in second moment, and sub-Gaussian with sub-Gaussian norm bounded by a universal constant, there exist universal constants $C_V,c_v>0$ such that, for every $t\ge0$,
\[
s_{\min}(Z)
\ge
\sqrt m-C_V\sqrt n-t
\]
with probability at least $1-2\exp(-c_Vt^2)$, where $s_{\min}(Z)$ denotes the smallest singular value of $Z$.
Take $t=\sqrt m/4$. If
\[
m\ge 16C_V^2 n,
\]
then $C_V\sqrt n\le \sqrt m/4$, and hence
\[
s_{\min}(Z)\ge \frac{\sqrt m}{2}
\]
with probability at least $1-2\exp(-c_Vm/16)$.

On this event, for every $v\in\RR^n$,
\[
\|Xv\|_2
=
\|Z\Omega_n^{1/2}v\|_2
\ge
s_{\min}(Z)\|\Omega_n^{1/2}v\|_2
\ge
\frac{\sqrt m}{2}\sqrt{\kappa}\,\|v\|_2.
\]
Thus
\[
\lambda_{\min}(X^\top X)\ge \frac{\kappa m}{4}.
\]
In particular, $X^\top X$ is positive definite, so $\operatorname{rank}(X)=n$.
The result follows by taking
\[
\Gamma=16C_V^2,
\qquad
\xi=\frac{c_V}{16}.
\]
\end{proof}

\begin{proof}[Proof of Corollary~\ref{cor:random-design}]
Let
\[
\mathcal E_{m,n}
:=
\left\{
\operatorname{rank}(X)=n
\quad\text{and}\quad
\lambda_{\min}(X^\top X)\ge \frac{\kappa m}{4}
\right\}.
\]
By the tallness assumption in Corollary~\ref{cor:random-design}, we have $m\ge \Gamma n$ for all sufficiently large $m,n$, where $\Gamma$ is the universal constant from Lemma~\ref{lem:gaussian-random-design-conditioning}.
Therefore the lemma gives
\[
\PP_X(\mathcal E_{m,n}^c)\le 2\exp(-\xi m)\to0.
\]

To apply Theorem~\ref{thm:main} on the high-probability event $\mathcal E_{m,n}$, first define the deterministic lower bound
\[
\underline L_{\mathrm{noise},m,n}
:=
\max\left\{
\ell\in\mathbb N:
\frac{2^{3\ell/2}}{\ell^{3/2}}
\le
c_{\mathrm{noise}}\,\frac{\kappa m}{4\sigma^2 n^{3/2}}
\right\},
\]
and
\[
\underline L_{m,n}
:=
\min\{L_{\mathrm{Cast}},\underline L_{\mathrm{noise},m,n}\}.
\]
Because
\[
\frac{m}{\sigma^2 n^{3/2}}\to\infty
\]
and $L_{\mathrm{Cast}}\to\infty$, we have
\[
\underline L_{m,n}\to\infty.
\]
On $\mathcal E_{m,n}$,
\[
L(X)\ge \underline L_{m,n}.
\]

It follows that $L(X)\to\infty$ in $\PP_X$-probability.
Indeed, for any fixed $K\in\mathbb N$,
\[
\PP_X\{L(X)<K\}
\le
\PP_X(\mathcal E_{m,n}^c)
+
\mathbbm{1}\{\underline L_{m,n}<K\}
\longrightarrow 0,
\]
since $\PP_X(\mathcal E_{m,n}^c)\to0$ and $\underline L_{m,n}\to\infty$.

We also have $\varepsilon_n(L(X))\to0$ in $\PP_X$-probability.
The approximation term satisfies
\[
(b-a)^\alpha 2^{-\alpha L(X)}\to0
\]
in $\PP_X$-probability because $L(X)\to\infty$ in $\PP_X$-probability.
For the stochastic term, since $L(X)\le L_{\mathrm{Cast}}$,
\[
\sqrt{\frac{L(X)2^{L(X)}}{n}}
\le
\sqrt{\frac{L_{\mathrm{Cast}}2^{L_{\mathrm{Cast}}}}{n}}
\to0,
\]
where the last convergence follows from the definition of $L_{\mathrm{Cast}}$.
Hence
\[
\varepsilon_n(L(X))\to0
\]
in $\PP_X$-probability.

We now reduce the random-design statement to the fixed-design theorem.
For a realized design $X$, define
\[
a_{m,n}(X)
:=
\EE_{g_0}^X\!\left[
\Pi_{L(X),X}\!\left(
\|g-g_0\|_\infty>C_0\varepsilon_n(L(X))
\mid
\bY
\right)
\right],
\]
where $C_0>0$ is the constant from Theorem~\ref{thm:main}. 
We claim that
\[
\sup_{X\in\mathcal E_{m,n}} a_{m,n}(X)\to0.
\]
Indeed, if not, then there exist $\eta>0$, a subsequence $(m_j,n_j)$, and deterministic matrices $X_{m_j,n_j}\in\mathcal E_{m_j,n_j}$ such that
\[
a_{m_j,n_j}(X_{m_j,n_j})\ge \eta
\]
for all $j$. But along this deterministic sequence of designs, we have
\[
\operatorname{rank}(X_{m_j,n_j})=n_j
\]
and
\[
L(X_{m_j,n_j})\ge \underline L_{m_j,n_j}\to\infty.
\]
Therefore Theorem~\ref{thm:main}, applied to this deterministic design sequence, gives
\[
a_{m_j,n_j}(X_{m_j,n_j})\to0,
\]
a contradiction. Hence
\[
\sup_{X\in\mathcal E_{m,n}} a_{m,n}(X)\to0.
\]

Finally, since $0\le a_{m,n}(X)\le1$,
\[
\EE_X a_{m,n}(X)
\le
\sup_{X\in\mathcal E_{m,n}} a_{m,n}(X)
+
\PP_X(\mathcal E_{m,n}^c)
\to0.
\]
That is,
\[
\EE_X\EE_{g_0}^X\!\left[
\Pi_{L(X),X}\!\left(
\|g-g_0\|_\infty>C_0\varepsilon_n(L(X))
\mid
\bY
\right)
\right]
\to0.
\]
This proves Corollary~\ref{cor:random-design}.
\end{proof}
\end{document}